\documentclass[letterpaper, 10 pt, conference]{ieeeconf}  

\IEEEoverridecommandlockouts                              
\let\labelindent\relax
\usepackage{amsthm}
\usepackage{amsmath,amssymb,mathtools}
\usepackage{calrsfs}\DeclareMathAlphabet{\pazocal}{OMS}{zplm}{m}{n}
\usepackage{cite}
\usepackage[inline]{enumitem}
\usepackage{bm}
\usepackage{tikz}

\renewcommand{\baselinestretch}{0.93}

\newtheorem{assumption}{Assumption}
\newtheorem{theorem}{Theorem}
\newtheorem{lemma}{Lemma}
\newtheorem{proposition}{Proposition}
\newtheorem{remark}{Remark}

\title{\LARGE \bf
Joint~Parameter~and~State~Estimation~of~Noisy~Discrete-Time~Nonlinear~Systems:~A~Supervisory~Multi-Observer~Approach
}

\author{T.J. Meijer$\,^{1,\dagger}$, V.S. Dolk$\,^2$, M.S. Chong$\,^1$, R. Postoyan$\,^3$, B. de Jager$\,^1$, D. Nešić$\,^4$ and W.P.M.H. Heemels$\,^1$
\thanks{$^1\,$Tomas Meijer, Michelle Chong, Bram de Jager and Maurice Heemels are with the Department of Mechanical Engineering, Eindhoven University of Technology, P.O. Box 513, 5600 MB Eindhoven, The Netherlands. {\tt\small \{t.j.meijer; m.s.t.chong; a.g.de.jager; m.heemels\}@tue.nl}}%
\thanks{$^2\,$Victor Dolk is with ASML, De Run 6665, 5504 DT Veldhoven, The Netherlands. {\tt\small victor.dolk@asml.com}}%
\thanks{$^3\,$Romain Postoyan is with the Université de Lorraine, CNRS, CRAN, F-54000 Nancy, France. {\tt\small romain.postoyan@univ-lorraine.fr}}%
\thanks{$^4\,$Dragan Nešić is with the Department of Electrical and Electronic Engineering, the University of Melbourne, Parkville, VIC 3010, Australia. {\tt\small dnesic@unimelb.edu.au}}
\thanks{$^\dagger$ Corresponding author: T.J. Meijer.}}
\date{\today}

\begin{document}

\maketitle
\thispagestyle{empty}
\pagestyle{empty}

\begin{abstract}
	This paper presents two schemes to jointly estimate parameters and states of discrete-time nonlinear systems in the presence of bounded disturbances and noise. The parameters are assumed to belong to a known compact set. Both schemes are based on sampling the parameter space and designing a state observer for each sample. A supervisor selects one of these observers at each time instant to produce the parameter and state estimates. In the first scheme, the parameter and state estimates are guaranteed to converge within a certain margin of their true values in finite time, assuming that a sufficiently large number of observers is used and a persistence of excitation condition is satisfied in addition to other observer design conditions. This convergence margin is constituted by a part that can be chosen arbitrarily small by the user and a part that is determined by the noise levels. The second scheme exploits the convergence properties of the parameter estimate to perform subsequent zoom-ins on the parameter subspace to achieve stricter margins for a given number of observers. The strengths of both schemes are demonstrated using a numerical example.
\end{abstract}

\section{Introduction}\label{sec:introduction}
\PARstart{J}{oint} parameter and state estimation is a highly relevant problem in many applications, such as synchronization of digital twins with their physical counterparts, see, e.g.,~\cite{Rasheed2020}, and sensor or source localization (in distributed parameter systems), see, e.g.,~\cite{Boerlage2008,Atanasov2014,Sawo2009}. In many cases such combined estimation problems arise, even when the aim is to estimate only the parameters of a system, as a result of the full state either not being measurable and/or measurements being corrupted by noise. A common approach to the joint parameter and state estimation problem is to augment the state with the parameters (and add constant parameter dynamics) and formulate it as a state estimation problem~\cite{Sarrka2013}. The state of the resulting system is then estimated using nonlinear state estimation algorithms, such as nonlinear Kalman filters or particle filters~\cite{Sarrka2013}, however, in general the underlying structure of the original model is lost leading to a (highly) nonlinear state estimation problem. For example, the augmented state approach turns joint estimation of an uncertain linear system with affine parameter dependencies into a bilinear state estimation problem. Following this path, it is typically difficult to provide convergence results~\cite{Besancon2006}. Joint parameter and state estimation schemes that do provide analytical convergence results often apply only to specific classes of systems, see, e.g.,~\cite{Besancon2006,Farza2009,Alvaro-Mendoza2020}.

\begin{figure}[!t]
	\centering
	\includegraphics[width=.41\textwidth]{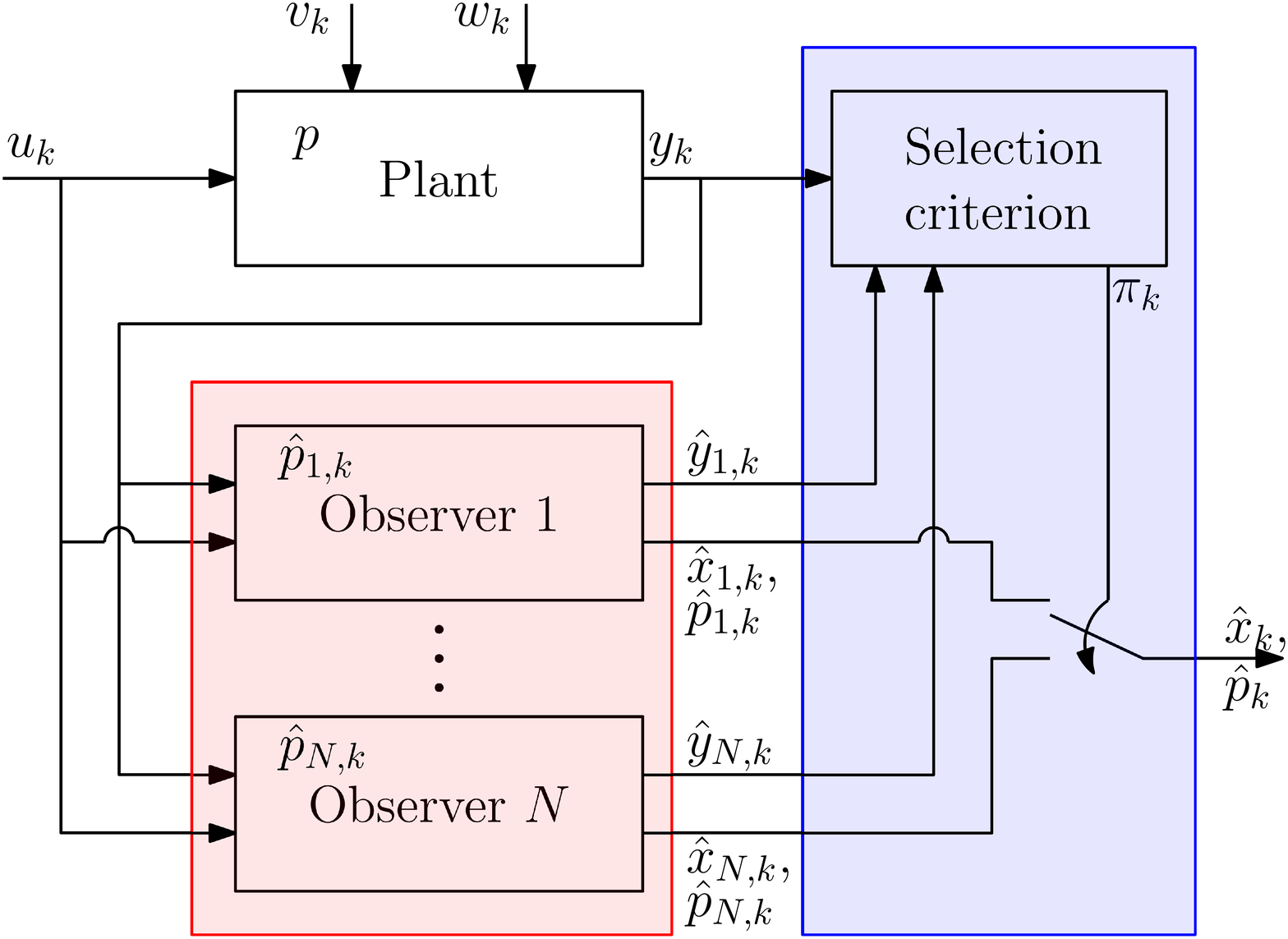}
	\caption{Supervisory observer scheme consisting of the~(\protect\tikz\protect\filldraw[fill=red!10!,draw=red] (0,0) rectangle (0.15,0.15);)~multi-observer and (\protect\tikz\protect\filldraw[fill=blue!10!,draw=blue] (0,0) rectangle (0.15,0.15);) supervisor. Relevant notation is introduced in Section~\ref{sec:static-sampling}.}
	\label{fig:multi-observer}
\end{figure}

In this paper, the joint parameter and state estimation problem for discrete-time nonlinear systems in the presence of bounded process and measurement noise is addressed in a different way. We exploit a supervisory observer framework that was recently developed in~\cite{Chong2015} for \textit{continuous-time} systems and \textit{without the consideration of disturbances and noise}. It is assumed that the parameters are constant and belong to a known compact set, with no restriction on its "size". The so-called \textit{supervisory observer} scheme, as depicted in Fig.~\ref{fig:multi-observer}, consists of \begin{enumerate*}[label=(\alph*)]
	\item the \textit{multi-observer}, a bank of multiple state observers--each designed for a parameter value sampled from the known parameter set--and
	\item the \textit{supervisor}, which at any given time instant selects one of the observers to provide the state and parameter estimates.
\end{enumerate*} Such multi-observer schemes have also been proved useful for many other purposes, such as, safeguarding systems against sensor attacks, see, e.g.,~\cite{Chong2020}, and the context of adaptive control, see, e.g.,~\cite{Hespanha2001}. An advantage of this sampling-based approach compared to the augmented state-space approach is that, for each parameter sample, the structure of the underlying system is preserved. This fact allows us to employ observers tailored to the specific model structure, which come with certain convergence guarantees and convenient (LMI-based) synthesis procedures, see, e.g.,~\cite{Heemels2010} for LPV systems or~\cite{Ibrir2007} for a class of nonlinear systems. The convergence properties of the individual state observers in the multi-observer are combined with a persistence of excitation (PE) condition to arrive at convergence guarantees for the supervisory observer. To be more concrete, the parameter and state estimates are guaranteed to converge within a certain margin of their true values, given that a sufficiently large number of observers is used. This sampling-based approach, which uses a static sampling policy, is rather simple to implement, but the number of samples (and, hence, the number of observers running in parallel) required to guarantee that the parameter error converges to within a given margin grows exponentially with the dimension of the parameter space. This inspired the development of a second scheme, which exploits the convergence result to iteratively zoom in by resampling from a shrinking subspace of the original parameter space. The resulting dynamic sampling policy is able to, for a given number of observers, guarantee tighter bounds on the parameter and state estimates. Alternatively, the dynamic scheme can be used to achieve a given margin of convergence using fewer observers than the static scheme. 

The extension of the continuous-time results in~\cite{Chong2015} to discrete-time is motivated by the fact that real-time implementation of any estimation algorithm requires discretization and that measurements become available at discrete time instances. Additionally, the discrete-time formulation enables parameter and state estimation of systems in feedback interconnection with a discrete-time control architecture such as model predictive control. The inclusion of process and measurement noise in the supervisory multi-observer framework is another major contribution, which allows us to provide more realistic performance guarantees for the proposed estimator, that was not addressed in~\cite{Chong2015,Chong2017}. However, it poses additional technical challenges including distinguishing between the effects of noise and parameter errors on our state and output estimation errors. In fact, this is only possible to some extent and, unlike in the noiseless case, the parameter error cannot be made arbitrarily small by using sufficiently many observers. Moreover, the dynamic sampling policy has to take into account the noise levels when zooming in, requiring a careful analysis. The strength of our framework is demonstrated on a numerical case study in the presence of noise.

The content of the paper is organized as follows. The problem definition is given in Section~\ref{sec:problem-definition}. Section~\ref{sec:static-sampling} presents the discrete-time supervisory observer using a static sampling policy. In Section~\ref{sec:dynamic-sampling}, the supervisory observer is adapted to utilize dynamic sampling. Finally, a numerical case study and
conclusions are given in Sections~\ref{sec:case-study} and~\ref{sec:conclusions}. All proofs can be found in the Appendix\\

\noindent\textbf{Notation.} Let $\mathbb{R}=\left(-\infty,\infty\right)$, $\mathbb{R}_{\geqslant 0} = \left[0,\infty\right)$, $\mathbb{R}_{>0}=\left(0,\infty\right)$, $\mathbb{N}=\left\{0,1,\hdots\right\}$, $\mathbb{N}_{\left[n,m\right]}=\left\{n,n+1,\hdots,m\right\}\subset\mathbb{N}$ for $n,m\in\mathbb{N}$ and $\mathbb{N}_{\geqslant n}=\left\{n,n+1,\hdots\right\}\subseteq\mathbb{N}$ for $n\in\mathbb{N}$. Moreover, $\left\|\cdot\right\|_p\,:\,\mathbb{R}^n\rightarrow\mathbb{R}_{\geqslant 0}$ with $n\in\mathbb{N}_{\geqslant 1}$ denotes an arbitrary (but the same throughout the paper) $p$-norm on $\mathbb{R}^n$, and we omit the subscript $p$ in the following, i.e., $\left\|\cdot\right\|=\left\|\cdot\right\|_p$. Let $\mathbb{B}^n\left(\xi,r\right)\coloneqq\left\{x\in\mathbb{R}^{n}\middle|\left\|x-\xi\right\|\leqslant r\right\}$ for $n\in\mathbb{N}_{\geqslant 1}$ represent the ball centered at $\xi\in\mathbb{R}^{n}$ of ``radius'' $r\in\mathbb{R}_{\geqslant 0}$ and let $\mathbb{B}_r^n=\mathbb{B}^n\left(0,r\right)$ denote such a set centered at the origin. For a sequence $\left\{x_k\right\}_{k\in\mathbb{N}}$ with $x_k\in\mathbb{R}^{n}$ and $n\in\mathbb{N}$, we denote $\left\|\left\{x_k\right\}\right\| = \left\|\left\{x_k\right\}_{k\in\mathbb{N}}\right\|_p\coloneqq \sup_{k\in\mathbb{N}}\left\|x_k\right\|_p$ where the subscript $p$ is again omitted for the sake of compactness. The space of all bounded sequences taking values in $\mathbb{R}^n$ with $n\in\mathbb{N}$ is denoted $\ell^{\infty}\coloneqq\{\{x_k\}_{k\in\mathbb{N}}\,|\,\|\{x_k\}\|_\infty <\infty\}$. The notation $\left(u,v\right)$ stands for $[u^\top\, v^\top]^\top$, where $u\in\mathbb{R}^m$ and $v\in\mathbb{R}^{n}$ with $\left(n,m\right)\in\mathbb{N}^2$. A continuous function $\alpha\,:\,\mathbb{R}_{\geqslant 0}\rightarrow\mathbb{R}_{\geqslant 0}$ is a $\pazocal{K}$-function ($\alpha\in\pazocal{K}$) if it is strictly increasing and $\alpha\left(0\right)=0$. If, in addition, $\alpha\left(r\right)\rightarrow\infty$ as $r\rightarrow\infty$, then $\alpha$ is a $\pazocal{K}_\infty$-function ($\alpha\in\pazocal{K}_{\infty}$). A continuous function $\beta\,:\,\mathbb{R}_{\geqslant 0}\times\mathbb{R}_{\geqslant 0}\rightarrow\mathbb{R}_{\geqslant 0}$ is a $\pazocal{KL}$-function ($\beta\in\pazocal{KL}$) if $\beta\left(\cdot,s\right)\in\pazocal{K}$ for each $s\in\mathbb{R}_{\geqslant 0}$, $\beta\left(r,\cdot\right)$ is non-increasing and $\beta\left(r,s\right)\rightarrow 0$ as $s\rightarrow\infty$ for each $r\in\mathbb{R}_{\geqslant 0}$.

\section{Problem definition}\label{sec:problem-definition}
Consider the discrete-time system given by
\begin{subequations}
	\begin{alignat}{2}
		x_{k+1} &= f\left(x_k,p,u_k,v_k\right),\\
		y_k &= h\left(x_k,p,u_k,w_k\right),
	\end{alignat}
	\label{eq:system}%
\end{subequations}
where $x_k\in\mathbb{R}^{n_x}$, $u_k\in\mathbb{R}^{n_u}$ and $y_k\in\mathbb{R}^{n_y}$ denote the state, input and output, respectively, at time instant $k\in\mathbb{N}$. In addition, the following assumptions are adopted.
\begin{assumption}
	\label{ass:bounded-sequences}
	The input $u_k$, process noise $v_k$ and measurement noise $w_k$ in~\eqref{eq:system} are bounded, i.e., there exist constants $\Delta_u,\Delta_v,\Delta_w\in\mathbb{R}_{\geqslant 0}$ such that for all $k\in\mathbb{N}$
	\begin{equation}
		u_k\in\mathbb{B}_{\Delta_u}^{n_u},\quad v_k\in\mathbb{B}_{\Delta_v}^{n_v}\quad\text{and}\quad w_k\in\mathbb{B}_{\Delta_w}^{n_w}.
	\end{equation}
\end{assumption}
\begin{assumption}
	\label{ass:compact-parameter-set}
		The parameter vector $p$ is constant and unknown and it belongs to a given compact set $\mathbb{P}$, i.e., $p\in\mathbb{P}\subset\mathbb{R}^{n_p}$.
\end{assumption}
\noindent Assumption~\ref{ass:bounded-sequences} means that $\left\{u_k\right\}_{k\in\mathbb{N}},\left\{v_k\right\}_{k\in\mathbb{N}},\left\{w_k\right\}_{k\in\mathbb{N}}\in\ell^{\infty}$, which is a reasonable assumption in practice. It should be noted that $\Delta_v$ and $\Delta_w$ in Assumption~\ref{ass:bounded-sequences} do not need to be known to implement the estimation schemes, their existence alone is sufficient. The input $u_k$ and output $y_k$ are known/measured, while the full state $x_k$, process noise $v_k$ and measurement noise $w_k$ are unknown. Moreover, the functions $f\,:\,\mathbb{R}^{n_x}\times\mathbb{P}\times\mathbb{R}^{n_u}\times\mathbb{R}^{n_v}\rightarrow\mathbb{R}^{n_x}$ and $h\,:\,\mathbb{R}^{n_x}\times\mathbb{P}\times\mathbb{R}^{n_u}\times\mathbb{R}^{n_w}\rightarrow\mathbb{R}^{n_y}$ are given and $h$ is locally Lipschitz continuous. For any initial condition $x_0\in\mathbb{R}^{n_x}$, input sequence $\left\{u_k\right\}_{k\in\mathbb{N}}$ with $u_k\in\mathbb{R}^{n_u}$, process noise sequence $\left\{v_k\right\}_{k\in\mathbb{N}}$ with $v_k\in\mathbb{R}^{n_v}$, measurement noise sequence $\left\{w_k\right\}_{k\in\mathbb{N}}$ with $w_k\in\mathbb{R}^{n_w}$ for $k\in\mathbb{N}$ and parameters $p\in\mathbb{P}$, the system~\eqref{eq:system} admits a unique solution	defined for all $k\in\mathbb{N}$. Finally, the following assumption is adopted.
\begin{assumption}\label{ass:bounded-state}
	The solutions to~\eqref{eq:system} are uniformly bounded, i.e., for all $\Delta_x,\Delta_u,\Delta_v,\Delta_w\in\mathbb{R}_{\geqslant 0}$, there exists a constant $K_x=K_x\left(\Delta_x,\Delta_u,\Delta_v,\Delta_w\right)\in\mathbb{R}_{>0}$ such that for all $x_0\in\mathbb{B}_{\Delta_x}^{n_x}$, $\left\{u_k\right\}_{k\in\mathbb{N}}$ with $u_k\in\mathbb{B}_{\Delta_u}^{n_u}$, $\left\{v_k\right\}_{k\in\mathbb{N}}$ with $v_k\in\mathbb{B}_{\Delta_v}^{n_v}$ and $\left\{w_k\right\}_{k\in\mathbb{N}}$ with $w_k\in\mathbb{B}_{\Delta_w}^{n_w}$ for any $k\in\mathbb{N}$, it holds that $x_k\in\mathbb{B}_{K_x}^{n_x}$ for all $k\in\mathbb{N}$.
\end{assumption}
\noindent The bound $K_x$ in Assumption~\ref{ass:bounded-state} does not need to be known to implement the proposed estimation algorithms, only its existence has to be ensured. 

Our objective is to jointly estimate the parameter vector $p$ and the state of the system~\eqref{eq:system} (within certain margins) subject to bounded process noise $v_k$ and measurement noise $w_k$, given the input $u_k$ and the measured output $y_k$.

\section{Supervisory observer: static sampling policy}\label{sec:static-sampling}
The parameter and state estimation schemes presented in this paper consist of two subsystems, as shown in Fig.~\ref{fig:multi-observer}. The first subsystem is the so-called multi-observer, which is a collection of observers that operate in parallel, where each observer is designed for a different parameter vector sampled from the parameter space. The second subsystem is a supervisor. The outputs of the observers are fed to the supervisor, which selects one of the observers based on a selection criterion and outputs its state estimate and corresponding parameter sample as the estimates produced by the overall estimation scheme. In this section, the parameter samples are obtained using a \textit{static} sampling policy meaning that these samples are fixed for all times. Later, in Section~\ref{sec:dynamic-sampling}, we consider a \textit{dynamic} sampling policy, which aims to reduce the computational complexity of the estimation scheme.

\subsection{Multi-observer}
The parameter space $\mathbb{P}$ is sampled to produce $N$ parameter samples $\hat{p}_i\in\mathbb{P}$ for $i\in\pazocal{N}\coloneqq\mathbb{N}_{\left[1,N\right]}$. This sampling is performed in such a way that the maximum distance of the true parameter to the nearest sample tends to zero as $N$ tends to infinity, i.e.,
\begin{equation}
	\lim_{N\rightarrow\infty} \max_{p\in\mathbb{P}} \min_{i\in\pazocal{N}} \left\|\hat{p}_i-p\right\| = 0.
	\label{eq:sampling-scheme}
\end{equation}
This can be ensured, for instance, by employing a uniform sampling of the parameter space. For each $\hat{p}_i$, $i\in\pazocal{N}$, a state observer is designed, given by
\begin{subequations}
	\begin{alignat}{2}
		\hat{x}_{i,k+1} &= \hat{f}\left(\hat{x}_{i,k},\hat{p}_i,u_k,y_k\right),\\
		\hat{y}_{i,k} &= h\left(\hat{x}_{i,k},\hat{p}_i,u_k,0\right),
	\end{alignat}
	\label{eq:multi-observer}%
\end{subequations}
where $\hat{x}_{i,k}\in\mathbb{R}^{n_x}$ and $\hat{y}_{i,k}\in\mathbb{R}^{n_y}$ denote, respectively, the state and output estimate of the $i$-th observer at time $k\in\mathbb{N}$. The function $\hat{f}\,:\,\mathbb{R}^{n_x}\times\left\{\hat{p}_i\right\}_{i\in\pazocal{N}}\times\mathbb{R}^{n_u}\times\mathbb{R}^{n_y}\rightarrow\mathbb{R}^{n_x}$ is well-designed such that the solutions to~\eqref{eq:multi-observer} are  defined for all time $k\in\mathbb{N}$, any initial condition $\hat{x}_{i,0}\in\mathbb{R}^{n_x}$, input sequence $\left\{u_k\right\}_{k\in\mathbb{N}}$, output sequence $\left\{y_k\right\}_{k\in\mathbb{N}}$ and parameter sample $\hat{p}_i\in\mathbb{P}$, $i\in\pazocal{N}$.

Let $\tilde{x}_{i,k}\coloneqq\hat{x}_{i,k}-x_k$ denote the state estimation error, $\tilde{y}_{i,k}\coloneqq\hat{y}_{i,k}-y_k$ the output estimation error and $\tilde{p}_i\coloneqq \hat{p}_i-p$ the parameter estimation error of the $i$-th observer. Since $\mathbb{P}$ is compact, there exists a compact set $\mathbb{D}\subset\mathbb{R}^{n_p}$ such that $\tilde{p}_i\in\mathbb{D}$ for any $p,\hat{p}_i\in\mathbb{P}$. The state and output estimation errors are governed by
\begin{subequations}
	\begin{alignat}{2}
		\tilde{x}_{i,k+1} &= F\left(\tilde{x}_{i,k},x_k,\tilde{p}_i,p,u_k,v_k,w_k\right),\\
		\tilde{y}_{i,k} &= H\left(\tilde{x}_{i,k},x_k,\tilde{p}_i,p,u_k,w_k\right),\label{eq:error-system-output}%
	\end{alignat}
	\label{eq:error-system}%
\end{subequations}
where the functions $F\,:\,\mathbb{R}^{n_x}\times\mathbb{R}^{n_x}\times\mathbb{D}\times\mathbb{P}\times\mathbb{R}^{n_u}\times\mathbb{R}^{n_v}\times\mathbb{R}^{n_w}\rightarrow\mathbb{R}^{n_x}$ and $H\,:\,\mathbb{R}^{n_x}\times\mathbb{R}^{n_x}\times\mathbb{D}\times\mathbb{P}\times\mathbb{R}^{n_u}\times\mathbb{R}^{n_w}\rightarrow\mathbb{R}^{n_y}$ are given by $F\left(\tilde{x},x,\tilde{p},p,u,v,w\right) = \hat{f}\left(\tilde{x}+x,\tilde{p}+p,u,h\left(x,p,u,w\right)\right)-f\left(x,p,u,v\right)$ and $H\left(\tilde{x},x,\tilde{p},p,u,w\right) = h\left(\tilde{x}+x,\tilde{p}+p,u,0\right)- h\left(x,p,u,w\right)$. The observers~\eqref{eq:multi-observer} are assumed to be robust with respect to the parameter error and noise in the following sense.
\begin{assumption}\label{ass:iss-lyap}
	There exist functions $\alpha_1,\alpha_2,\alpha_3\in\pazocal{K}_\infty$ and a continuous non-negative function $\sigma\,:\,\mathbb{D}\times\mathbb{R}^{n_v}\times\mathbb{R}^{n_w}\times\mathbb{R}^{n_x}\times\mathbb{R}^{n_u}\rightarrow\mathbb{R}_{\geqslant 0}$ with $\sigma\left(0,0,0,x,u\right)=0$ for all $x\in\mathbb{R}^{n_x}$ and $u\in\mathbb{R}^{n_u}$ such that there exists a function $V\,:\,\mathbb{P}\times\mathbb{R}^{n_x}\rightarrow\mathbb{R}_{\geqslant 0}$, which satisfies, for all $\tilde{x},x\in\mathbb{R}^{n_x}$, $p,\hat{p}\in\mathbb{P}$, $u\in\mathbb{R}^{n_u}$, $v\in\mathbb{R}^{n_v}$ and $w\in\mathbb{R}^{n_w}$, that $V(\hat{p},\cdot)$ is continuous and
	\begin{subequations}
		\begin{alignat}{2}
			\alpha_1\left(\left\|\tilde{x}\right\|\right) &\leqslant V\left(\hat{p},	\tilde{x}\right) \leqslant \alpha_2\left(\left\|\tilde{x}\right\|\right),\label{eq:sandwich-bounds}\\
			V\left(\hat{p},\tilde{x}^+\right) &\leqslant V\left(\hat{p},\tilde{x}\right)-\alpha_3\left(\left\|\tilde{x}\right\|\right)+ \sigma\left(\tilde{p},v,w,x,u\right),\label{eq:diffV}%
		\end{alignat}
	\end{subequations}
	for $\tilde{x}^+ = F\left(\tilde{x},x,\tilde{p},p,u,v,w\right)$.
\end{assumption}
\noindent Assumption~\ref{ass:iss-lyap} implies that the error systems~\eqref{eq:error-system} corresponding to the observers in~\eqref{eq:multi-observer} are locally input-to-state stable (ISS) with respect to $\tilde{p}_i$, $v_k$ and $w_k$~\cite{Jiang2001}, as shown in Lemma~\ref{lem:ISS}. For linear uncertain systems, Luenberger observers satisfy Assumption~\ref{ass:iss-lyap} and, in Section~\ref{sec:case-study}, it is shown that a class of circle-criterion-based nonlinear observers also satisfies this assumption. 

\subsection{Supervisor}
At every time $k\in\mathbb{N}$, the supervisor selects one observer from the multi-observer. To be able to assess the accuracy of the different observers, the supervisor computes a monitoring signal for each observer, which, for $i\in\pazocal{N}$, is given by
\begin{equation}
	\mu_{i,k} = \sum_{j=0}^{k-1}\lambda^{k-1-j}\left\|\tilde{y}_{i,j}\right\|^2,\quad k\in\mathbb{N},
	\label{eq:monitoring-signals}%
\end{equation}
where $\lambda\in\left[0,1\right)$ is a design parameter. The $i$-th monitoring signal~\eqref{eq:monitoring-signals} can be implemented using the difference equation
\begin{equation}
	\mu_{i,k+1} = \lambda\mu_{i,k} + \left\|\tilde{y}_{i,k}\right\|^2,\quad k\in\mathbb{N},
\end{equation}
with the initial condition $\mu_{i,0}=0$. The output errors of the state observers are assumed to satisfy the following PE condition.
\begin{assumption}\label{ass:PE}
	For any $\Delta_{\tilde{x}},\Delta_x,\Delta_u,\Delta_v,\Delta_w\geqslant 0$, there exist a function $\alpha_{\tilde{y}}\in\pazocal{K}_{\infty}$ and an integer $N_{\mathrm{pe}} = N_{\mathrm{pe}}\left(\Delta_{\tilde{x}},\Delta_x,\Delta_u,\Delta_v,\Delta_w\right)\in\mathbb{N}_{\geqslant 1}$ such that for all $\tilde{p}_i\in\mathbb{D}$, $i\in\pazocal{N}$, $\tilde{x}_{i,0}\in\mathbb{B}^{n_x}_{\Delta_{\tilde{x}}}$, $p\in\mathbb{P}$, $x_0\in\mathbb{B}_{\Delta_x}^{n_x}$, $\left\{v_k\right\}_{k\in\mathbb{N}}$ with $v_k\in\mathbb{B}_{\Delta_v}^{n_v}$ and $\left\{w_k\right\}_{k\in\mathbb{N}}$ with $w_k\in\mathbb{B}_{\Delta_w}^{n_w}$ and for some input sequence $\left\{u_k\right\}_{k\in\mathbb{N}}$ with $u_k\in\mathbb{B}_{\Delta_u}^{n_u}$ for $k\in\mathbb{N}$, the solutions to~\eqref{eq:system} and~\eqref{eq:error-system} satisfy
	\begin{equation}
		\sum_{j=k-N_{\mathrm{pe}}}^{k-1}\left\|\tilde{y}_{i,j}\right\|^2\geqslant \alpha_{\tilde{y}}\left(\left\|\tilde{p}_i\right\|\right),\quad k\in\mathbb{N}_{\geqslant N_{\mathrm{pe}}}.
		\label{eq:PE}
	\end{equation}
\end{assumption}
\noindent Assumption~\ref{ass:PE} differs from the classical PE condition, see, e.g.,~\cite{Bitmead1984}, in that it considers solutions to~\eqref{eq:error-system-output} parametrized by $\tilde{p}_i$ and requires the sum in~\eqref{eq:PE} to grow with the norm of the parameter error. This ensures that the supervisor is able to infer quantitative information about the parameter estimation error of each state observer based on its monitoring signal.

At every time instant $k\in\mathbb{N}$, the supervisor selects (one of) the observer(s) with the smallest monitoring signal to obtain the estimates of $p$ and $x_k$. In the event that $\min_{i\in\pazocal{N}}\mu_{i,k}$ is not unique any observer from this subset can be chosen, resulting in a selection criterion where the index of the selected observer $\pi_k\,:\,\mathbb{N}\rightarrow\pazocal{N}$ satisfies
\begin{equation}
	\pi_k\in \arg\min_{i\in\pazocal{N}}\mu_{i,k},\quad k\in\mathbb{N}.
	\label{eq:selection-criterion}
\end{equation}
The resulting parameter estimate, state estimate and state estimation error at time $k\in\mathbb{N}$, denoted $\hat{p}_k$, $\hat{x}_k$ and $\tilde{x}_k$, respectively, are defined using $\pi_k$ as
\begin{equation}
		\hat{p}_k\coloneqq \hat{p}_{\pi_k},\quad\hat{x}_k\coloneqq \hat{x}_{\pi_k,k}\quad\text{and}\quad\tilde{x}_k\coloneqq\tilde{x}_{\pi_k,k}.
	\label{eq:estimates}%
\end{equation}

\subsection{Convergence guarantees}
The parameter and state estimates~\eqref{eq:estimates} converge to within certain margins of their true values $p$ and $x_k$ as stated in the following theorem.
\begin{theorem}
	\label{thm:static-sampling}
	Consider the system~\eqref{eq:system}, the multi-observer~\eqref{eq:multi-observer}, the monitoring signals~\eqref{eq:monitoring-signals}, the selection criterion~\eqref{eq:selection-criterion}, the parameter estimate, state estimate and state estimation error in~\eqref{eq:estimates}. Suppose Assumptions~\ref{ass:bounded-sequences}-\ref{ass:PE} hold. For any $\Delta_{\tilde{x}},\Delta_x,\Delta_u,\Delta_v,\Delta_w\in\mathbb{R}_{\geqslant 0}$ and any margins $\nu_{\tilde{p}},\nu_{\tilde{x}}\in\mathbb{R}_{>0}$, there exist functions $\upsilon_{\tilde{p}},\upsilon_{\tilde{x}},\omega_{\tilde{p}},\omega_{\tilde{x}}\in\pazocal{K}_{\infty}$, constant $K_{\tilde{x}}\in\mathbb{R}_{>0}$ and sufficiently large integers $N^{\star}\in\mathbb{N}_{\geqslant 1}$ and $M\in\mathbb{N}_{\geqslant 1}$ such that for any $N\in\mathbb{N}_{\geqslant N^{\star}}$, it holds for any $\tilde{x}_{i,0}\in\mathbb{B}_{\Delta_{\tilde{x}}}^{n_x}$, $i\in\pazocal{N}$, $p\in\mathbb{P}$, $x_0\in\mathbb{B}_{\Delta_x}^{n_x}$, $\left\{v_k\right\}_{k\in\mathbb{N}}$ with $v_k\in\mathbb{B}_{\Delta_v}^{n_v}$, $\left\{w_k\right\}_{k\in\mathbb{N}}$ with $w_k\in\mathbb{B}_{\Delta_w}^{n_w}$ and for some input sequence $\left\{u_k\right\}_{k\in\mathbb{N}}$ with $u_k\in\mathbb{B}_{\Delta_u}^{n_u}$ for $k\in\mathbb{N}$, which satisfies Assumption~\ref{ass:PE}, that $\tilde{x}_{i,k}\in\mathbb{B}_{K_{\tilde{x}}}^{n_x}$ for all $k\in\mathbb{N}$, $i\in\pazocal{N}$, and
	\begin{align}
			&\left\|\tilde{p}_{\pi_k}\right\| \leqslant \nu_{\tilde{p}} + \upsilon_{\tilde{p}}\left(\left\|\left\{v_j\right\}\right\|\right) + \omega_{\tilde{p}}\left(\left\|\left\{w_j\right\}\right\|\right),\quad \forall k\in\mathbb{N}_{\geqslant M},\nonumber\\
			&\limsup_{k\rightarrow \infty}\left\|\tilde{x}_k\right\| \leqslant \nu_{\tilde{x}} + \upsilon_{\tilde{x}}\left(\left\|\left\{v_j\right\}\right\|\right) + \omega_{\tilde{x}}\left(\left\|\left\{w_j\right\}\right\|\right).
		\label{eq:convergence-static}%
	\end{align}
\end{theorem}
\noindent The proof of Theorem~\ref{thm:static-sampling} is provided in the Appendix. In the noiseless case, i.e., $v_k=w_k=0$, the convergence margins can be made arbitrarily small since $\nu_{\tilde{p}}$ and $\nu_{\tilde{x}}$ can be made arbitrarily small by using sufficiently many observers. However, this is impossible in the presence of noise due to the terms in~\eqref{eq:convergence-static} depending on $\left\|\left\{v_k\right\}\right\|$ and $\left\|\left\{w_k\right\}\right\|$. 

\section{Supervisory observer: dynamic~sampling~policy}\label{sec:dynamic-sampling}
In this section we develop a dynamic sampling policy for joint parameter and state estimation of~\eqref{eq:system}. As stated in Theorem~\ref{thm:static-sampling}, when using a sufficiently large number of observers $N$, the parameter estimate converges to a given margin within a finite time. We exploit this result in the dynamic sampling policy to iteratively zoom in on the parameter subspace defined by the aforementioned margins through resampling. As a result, stricter bounds on the parameter and state estimates can be guaranteed compared to the static sampling policy for a given number of observers.

\subsection{Dynamic sampling policy}
Since the parameter set $\mathbb{P}$ is compact, there exist $p_c\in\mathbb{R}^{n_p}$ and $\Delta_0\in\mathbb{R}_{>0}$ such that 
\begin{equation}
	\mathbb{P}\subseteq \mathbb{B}^{n_p}\left(p_c,\Delta_0\right).
\end{equation}
Let $\nu\in\mathbb{R}_{>0}$ denote the desired bound on the parameter error, which either represents the required bound on the parameter error to guarantee asymptotic convergence of the state estimation error to within a desired margin or a desired bound imposed directly on the parameter estimation error. We also introduce a design parameter $\alpha\in(0,1)$, the so-called zooming factor, which determines the rate at which the considered parameter set shrinks. The dynamic sampling policy is initialized at $k=k_0=0$ by sampling $\mathbb{P}_0\coloneqq\mathbb{P}$, using a sampling scheme, which satisfies~\eqref{eq:sampling-scheme}, to obtain $N\in\mathbb{N}_{\geqslant 1}$ parameter samples $\hat{p}_{i,0}$, $i\in\pazocal{N}$. Here, $N$ is chosen sufficiently large such that, by Theorem~\ref{thm:static-sampling}, it holds for sufficiently large $M\in\mathbb{N}_{\geqslant 1}$ that $\left\|\hat{p}_{\pi_k,0}-p\right\| \leqslant C$ for all $k\in\mathbb{N}_{\geqslant M}$ with
\begin{equation}
	C\coloneqq \max\left\{\nu,\alpha\Delta_0\right\} + \upsilon_{\tilde{p}}\left(\left\|\left\{v_j\right\}\right\|\right) + \omega_{\tilde{p}}\left(\left\|\left\{w_j\right\}\right\|\right).
\end{equation} 
As a consequence, for $k\in\mathbb{N}_{\geqslant M}$, either the desired margin is achieved or $p\in \mathbb{P}_1\coloneqq\mathbb{B}^{n_p}(\hat{p}_{\pi_{k},0},\alpha\Delta_0+ \upsilon_{\tilde{p}}(\|\{v_j\}\|)+\omega_{\tilde{p}}(\|\{w_j\}\|))\cap\mathbb{P}_0$. Both cases cannot be distinguished on-line since the true parameter is unknown. Therefore, at $k=k_1$ with $k_1\in\mathbb{N}_{\geqslant M}$, even if the desired margin has already been achieved, the set $\mathbb{P}_1$ is sampled to obtain $N$ new samples $\left\{\hat{p}_{i,1}\right\}_{i\in\pazocal{N}}$. This procedure is performed iteratively at every $k_m$, $m\in\mathbb{N}$, with
\begin{equation}
	M_d \coloneqq k_{m+1} - k_m,
	\label{eq:intersampling-time}
\end{equation}
where $M_d\in\mathbb{N}_{\geqslant \max\left\{1,M\right\}}$ denotes the number of time steps between subsequent zoom-ins. The shrinking parameter subset $\mathbb{P}_m$, $m\in\mathbb{N}$, is defined recursively by
\begin{align}
	\mathbb{P}_{m+1} \coloneqq \mathbb{D}_m\cap\mathbb{P}_{m}
	\label{eq:shrinking-parameter-sets}
\end{align}
with $\mathbb{P}_0 = \mathbb{P}$, $\Delta_m\coloneqq \alpha\Delta_{m-1} = \alpha^m\Delta_0$ and $\mathbb{D}_m = \mathbb{B}^{n_p}\left(\hat{p}_{\pi_{k_m+1},m},\Delta_{m+1} + \upsilon_{\tilde{p}}\left(\left\|\left\{v_j\right\}\right\|\right) + \omega_{\tilde{p}}\left(\left\|\left\{w_j\right\}\right\|\right)\right)$. The spaces $\mathbb{P}_m$, $m\in\mathbb{N}$, are sampled in such a way that
\begin{equation}
	\max_{p\in\mathbb{P}_m}\min_{i\in\pazocal{N}} \left\|\hat{p}_{i,m}-p\right\| \leqslant \rho\left(\Delta_m,N\right)
	\label{eq:dynamic-sampling}
\end{equation}
with $\rho\in\pazocal{KL}$ and where $\left\{\hat{p}_{i,m}\right\}_{i\in\pazocal{N}}$ denote the obtained samples. The corresponding parameter errors are denoted $\tilde{p}_{i,m} = \hat{p}_{i,m}-p$, $i\in\pazocal{N}$ and $m\in\mathbb{N}$. It is worth mentioning that once the desired margin is achieved the algorithm still keeps zooming in and it can occur that, after zooming in a certain number of times, the subset that is being sampled no longer contains the true parameter. Regardless, the true parameter still lies within the desired margin of the selected parameter estimate and the convergence guarantees provided in this section remain valid. 

The dynamic sampling policy is incorporated into the multi-observer by designing state observers for each parameter sample $\hat{p}_{i,m}$, $i\in\pazocal{N}$, for the time instance $k\in\mathbb{N}_{\left[k_m,k_{m+1}-1\right]}$, $m\in\mathbb{N}$. The $i$-th state observer is given by
\begin{subequations}
	\begin{alignat}{2}
		\hat{x}_{i,k+1} &= \hat{f}(\hat{x}_{i,k},\hat{p}_{i,m},u_k,y_k),\\
		\hat{y}_{i,k} &= h(\hat{x}_{i,k},\hat{p}_{i,m},u_k,0),
	\end{alignat}
	\label{eq:dynamic-multi-observer}%
\end{subequations}
for $k\in\mathbb{N}_{\left[k_m,k_{m+1}-1\right]}$, $m\in\mathbb{N}$. Here, $\hat{f}\,:\,\mathbb{R}^{n_x}\times\left\{\hat{p}_{i,m}\right\}_{i\in\pazocal{N},m\in\mathbb{N}}\times\mathbb{R}^{n_u}\times\mathbb{R}^{n_y}\rightarrow\mathbb{R}^{n_x}$ is well-designed such that the solutions to~\eqref{eq:dynamic-multi-observer} are defined for all $k\in\mathbb{N}$ and any initial condition $\hat{x}_{i,0}\in\mathbb{R}^{n_x}$, input sequence $\left\{u_k\right\}_{k\in\mathbb{N}}$, output sequence $\left\{y_k\right\}_{k\in\mathbb{N}}$ and parameters $\hat{p}_{i,m}\in\mathbb{P}$, $i\in\pazocal{N}$ and $m\in\mathbb{N}$. We assume these observers satisfy Assumption~\ref{ass:iss-lyap}.

The dynamic sampling policy requires the monitoring signals used by the supervisor to be redefined. The redefined monitoring signals are reset upon resampling, i.e.,
\begin{equation}
	\mu_{i,k+1} = \begin{cases}
		\left\|\tilde{y}_{i,k}\right\|^2,\quad &k\in\left\{k_m\right\}_{m\in\mathbb{N}},\\
		\lambda\mu_{i,k} + \left\|\tilde{y}_{i,k}\right\|^2,\quad &k\in\mathbb{N}\setminus\left\{k_m\right\}_{m\in\mathbb{N}},
	\end{cases}
	\label{eq:dynamic-monitoring-signals}
\end{equation}
for $i\in\pazocal{N}$, with $\lambda\in\left[0,1\right)$. As before, the supervisor selects an observer from the multi-observer~\eqref{eq:dynamic-multi-observer} using the signal $\pi_k$ as defined in~\eqref{eq:selection-criterion}. The definition of the state estimate $\hat{x}_{k}$ and corresponding error $\tilde{x}_k$ in~\eqref{eq:estimates} are unchanged, however, the parameter estimate and corresponding error are redefined as
\begin{equation}
	\hat{p}_k\coloneqq \hat{p}_{\pi_k,m}\text{ and }\tilde{p}_{k} \coloneqq \tilde{p}_{\pi_k,m},\, \text{for }k\in\mathbb{N}_{\left[k_m+1,k_{m+1}\right]}.
	\label{eq:dynamic-parameter-estimate}
\end{equation}

\subsection{Convergence guarantees}
The parameter and state estimates produced by the supervisory observer using a dynamic sampling scheme satisfy similar convergence guarantees as in the static sampling case. This is stated in the following theorem for which the proof is provided in the Appendix.
\begin{theorem}
	\label{thm:dynamic-sampling}
	Consider the system~\eqref{eq:system}, the multi-observer~\eqref{eq:dynamic-multi-observer}, the monitoring signals~\eqref{eq:dynamic-monitoring-signals}, the selection criterion~\eqref{eq:selection-criterion}, the parameter estimate and corresponding error~\eqref{eq:dynamic-parameter-estimate}, the state estimate and corresponding error in~\eqref{eq:estimates} and the dynamic sampling policy~\eqref{eq:intersampling-time}-\eqref{eq:shrinking-parameter-sets}. Suppose Assumptions~\ref{ass:bounded-sequences}-\ref{ass:PE} hold. For any $\Delta_{\tilde{x}},\Delta_x,\Delta_u,\Delta_v,\Delta_w\in\mathbb{R}_{\geqslant 0}$, any margins $\nu_{\tilde{p}},\nu_{\tilde{x}}\in\mathbb{R}_{>0}$ and zooming factor $\alpha\in\left(0,1\right)$, there exist functions $\upsilon_{\tilde{p}},\upsilon_{\tilde{x}},\omega_{\tilde{p}},\omega_{\tilde{x}}\in\pazocal{K}_{\infty}$, scalar $K_{\tilde{x}}\in\mathbb{R}_{>0}$ and sufficiently large integers $M^\star\in\mathbb{N}_{\geqslant 1}$, $\bar{M}\in\mathbb{N}_{\geqslant 1}$ and $N^\star\in\mathbb{N}_{\geqslant 1}$ such that for any $N\in\mathbb{N}_{\geqslant N^\star}$ and $M_d\in\mathbb{N}_{\geqslant M^{\star}}$, it holds for any $\tilde{x}_{i,0}\in\mathbb{B}_{\Delta_{\tilde{x}}}^{n_x}$, $i\in\pazocal{N}$, $p\in\mathbb{P}$, $x_0\in\mathbb{B}_{\Delta_x}^{n_x}$, $\left\{v_k\right\}_{k\in\mathbb{N}}$ with $v_k\in\mathbb{B}_{\Delta_v}^{n_v}$, $\left\{w_k\right\}_{k\in\mathbb{N}}$ with $w_k\in\mathbb{B}_{\Delta_w}^{n_w}$ and for some input sequence $\left\{u_k\right\}_{k\in\mathbb{N}}$ with $u_k\in\mathbb{B}_{\Delta_u}^{n_u}$ for $k\in\mathbb{N}$, which satisfies Assumption~\ref{ass:PE}, that $\tilde{x}_{i,k}\in\mathbb{B}_{K_{\tilde{x}}}^{n_x}$ for all $k\in\mathbb{N}$, $i\in\pazocal{N}$, and
	\begin{align}
			&\left\|\tilde{p}_k\right\| \leqslant \nu_{\tilde{p}} + \upsilon_{\tilde{p}}\left(\left\|\left\{v_j\right\}\right\|\right) + \omega_{\tilde{p}}\left(\left\|\left\{w_j\right\}\right\|\right),\quad \forall k\in\mathbb{N}_{\geqslant M},\nonumber\\
			&\limsup_{k\rightarrow \infty}\left\|\tilde{x}_k\right\| \leqslant \nu_{\tilde{x}} + \upsilon_{\tilde{x}}\left(\left\|\left\{v_j\right\}\right\|\right) + \omega_{\tilde{x}}\left(\left\|\left\{w_j\right\}\right\|\right).
		\label{eq:convergence-dynamic}%
	\end{align}
\end{theorem}
\noindent Theorem~\ref{thm:dynamic-sampling} ensures the same guarantees as in Theorem~\ref{thm:static-sampling}, but it typically requires less observers to do so using the dynamic sampling policy, as will be illustrated in Section~\ref{sec:case-study}.
\begin{remark}
	The dynamic sampling policy in this paper uses a fixed number of samples, however, an alternative policy using the DIRECT algorithm which adds samples on-line is proposed for the continuous-time setting in~\cite{Chong2017}. This eliminates the need to estimate the required number of observers a-priori, which can be challenging. 
\end{remark}

\section{Case study}\label{sec:case-study}
In this section, we apply the results of Theorems~\ref{thm:static-sampling} and~\ref{thm:dynamic-sampling} to estimate the parameters and states of an example within the class of nonlinear systems given by
\begin{align}
		x_{k+1} &= A\left(p\right)x_k + G\left(p\right)\phi\left(Hx_k\right) + B\left(p\right)\left(u_k+v_k\right),\nonumber\\
		y_k &= Cx_k + w_k,
		\label{eq:illustrative-example}
\end{align}
where $x_k\in\mathbb{R}^{n_x}$, $u_k\in\mathbb{R}^{n_u}$, $v_k\in\mathbb{R}^{n_v}$, $w_k\in\mathbb{R}^{n_w}$ and $p\in\mathbb{R}^{n_p}$. Suppose Assumptions~\ref{ass:bounded-sequences}-\ref{ass:bounded-state} hold and $A(p)$, $B(p)$ and $G(p)$ are continuous in $p$ on $\mathbb{P}$. The nonlinearity $\phi\,:\,\mathbb{R}^{n_\phi}\rightarrow\mathbb{R}^{n_\phi}$ is such that $\phi(v) = (\phi_1(v_1),\hdots,\phi_{n_\phi}(v_{n_\phi}))$ for $v=(v_1,\hdots,v_{n_\phi})\in\mathbb{R}^{n_\phi}$ and there exist constants $\ell_{i}\in\mathbb{R}_{>0}$, $i\in\mathbb{N}_{\left[1,n_\phi\right]}$, such that, for all $v\in\mathbb{R}$, we have
\small\begin{equation}
	0\leqslant \frac{\partial\phi_i(v)}{\partial v}\leqslant \ell_i.
	\label{eq:lipschitz}
\end{equation}\normalsize
For $\hat{p}\in\mathbb{P}$, a state observer of the form~\cite{Ibrir2007,Yang2018}
\begin{subequations}
	\begin{alignat}{2}
		\hat{x}_{k+1} &= A(\hat{p})\hat{x}_{k}+ G(\hat{p})\phi(H\hat{x}_{k}+K(\hat{p})(C\hat{x}_{k}-y_k))\nonumber\\
		&\quad + B(\hat{p})u_k + L(\hat{p})(C\hat{x}_{k}-y_k),\\
		\hat{y}_{k} &= C\hat{x}_{k},
	\end{alignat}
	\label{eq:circle-criterion-observer}%
\end{subequations}
is designed by synthesizing observer matrices $K\left(\hat{p}\right)$ and $L\left(\hat{p}\right)$ such that the following proposition applies.
\begin{proposition}\label{prop:circle-criterion-observer}
	Consider the system~\eqref{eq:illustrative-example} and state observer~\eqref{eq:circle-criterion-observer}. Suppose there exist $P=P^\top\in\mathbb{R}^{n_x\times n_x}$, $M=\mathrm{diag}(m_1,\hdots,m_{n_\phi})$ with $m_i\in\mathbb{R}_{>0}$, $i\in\mathbb{N}_{\left[1,n_\phi\right]}$, and $\kappa_{\tilde{x}},\kappa_v,\kappa_w\in\mathbb{R}_{>0}$, such that $P\succ 0$ and, for all $p,\hat{p}\in\mathbb{P}$,
	\small\begin{equation}
		\begin{bmatrix}
			-P & \star & \star & \star & \star\\
			-\bm{A}^\top(\hat{p})P & \frac{\kappa_{\tilde{x}}}{2}I-\frac{1}{2}P & \star & \star & \star\\
			-G^\top(\hat{p})P & \frac{1}{2}M\bm{H}(\hat{p}) & -M\Lambda^{-1} & \star & \star\\
			B^\top(p)P & 0 & 0 & -\frac{\kappa_v}{2}I & \star\\
			L^\top(\hat{p})P & 0 & -\frac{1}{2}K^\top(\hat{p})M & 0 & -\frac{\kappa_w}{2}I
		\end{bmatrix}\preccurlyeq 0,
		\label{eq:design-inequality}
	\end{equation}\normalsize
	where $\bm{A}(\hat{p}) = A(\hat{p})+L(\hat{p})C$, $\bm{H}(\hat{p}) = H+K(\hat{p})C$ and $\Lambda = \mathrm{diag}(\ell_1,\hdots,\ell_{n_{\phi}})$, then Assumption~\ref{ass:iss-lyap} is satisfied.
\end{proposition}
\noindent The proof for Proposition~\ref{prop:circle-criterion-observer} is provided in the Appendix. The condition in~\eqref{eq:design-inequality} represents infinitely many linear matrix inequalities (LMIs) in $P$, $PL(\hat{p})$, $M$, $MK(\hat{p})$, $\kappa_{\tilde{x}}$, $\kappa_v$ and $\kappa_w$, due to its dependence on $\hat{p}$ and $p$. In order to solve~\eqref{eq:design-inequality}, it either needs to be discretized or, as we will see in our case study, sometimes structure can be exploited to reduce~\eqref{eq:design-inequality} to a finite number of LMIs. If Proposition~\ref{prop:circle-criterion-observer} and Assumption~\ref{ass:PE} apply, then Theorems~\ref{thm:static-sampling} or~\ref{thm:dynamic-sampling} hold, respectively, when the static or dynamic sampling policy is used.

Consider the system~\eqref{eq:illustrative-example} with the following matrices
\small\begin{align}
	A(p) &= \begin{bmatrix}
		1 & T_s\\
		0 & 1
	\end{bmatrix} - p\begin{bmatrix}
		\frac{1}{2} & \frac{1}{2}\\
		1 & 1
	\end{bmatrix}T_s,\, G(p) = p\begin{bmatrix}
		\frac{T_s}{2}\\
		T_s
	\end{bmatrix},\nonumber\\
	H^\top &=\begin{bmatrix}
		1\\
		1
	\end{bmatrix},\,B(p)=\begin{bmatrix}T_s\\ T_s\end{bmatrix}+p\begin{bmatrix}T_s\\-T_s\end{bmatrix}\text{ and }C = I,
	\label{eq:system-matrices}
\end{align}\normalsize
which is obtained by discretizing a continuous-time system, see~\cite{Ibrir2007}, with sampling time $T_s=0.01$. The nonlinearity in~\eqref{eq:illustrative-example} is given by
\begin{equation}
	\phi(v) = v+\sin(v),
\end{equation}
which satisfies~\eqref{eq:lipschitz} with Lipschitz constant $\ell_1=2$. Moreover, the parameter $p$ belongs to $\mathbb{P}\coloneqq \left[1,50\right]$. This example is a variation on~\cite[Example 1]{Ibrir2007} and~\cite[Example 1]{Yang2018} where we included process and measurement noise and an additional parameter dependency in $B(p)$. Notice that the system matrices~\eqref{eq:system-matrices} all depend affinely on the unknown parameter. If we restrict the observer matrices $L(\hat{p})$ and $K(\hat{p})$ to also be affine in $p$, i.e.,
\begin{equation}
	L(\hat{p}) = L_0+\hat{p}L_1\quad\text{and}\quad K(\hat{p}) = K_0 + \hat{p}K_1,
	\label{eq:affine-observer-matrices}
\end{equation}
with $L_i\in\mathbb{R}^{n_x\times n_y}$ and $K_i\in\mathbb{R}^{n_\phi\times n_y}$, $i=0,1$, the LMI in~\eqref{eq:design-inequality} becomes affine in $(p,\hat{p})\in\mathbb{P}\times \mathbb{P}$. Since $\mathbb{P}\times\mathbb{P}$ is convex, the condition~\eqref{eq:design-inequality} is satisfied for all $p,\hat{p}\in\mathbb{P}$ if and only if it is satisfied at each of the $(2n_p)^2 = 4$ vertices~\cite{Scherer2011}. We set $\kappa_{\tilde{x}}=0.1$ and minimize $\kappa_v+5\kappa_w$ subject to~\eqref{eq:design-inequality}, for all $p,\hat{p}\in\left\{1,50\right\}$, by means of the MATLAB toolbox \textsc{YALMIP}~\cite{Lofberg2004} together with the external solver \textsc{MOSEK}~\cite{Mosek2019}. Restricting ourselves to a Lyapunov function, which is independent of $\hat{p}$ and to affine observer matrices~\eqref{eq:affine-observer-matrices} introduces conservatism compared to, for instance, sampling the parameter space and then solving the LMIs. However, it has the advantage that resampling in the dynamic sampling policy is computationally efficient as it only requires evaluating~\eqref{eq:affine-observer-matrices} without solving LMIs on-line.

\begin{figure}[!tb]
	\centering
	\includegraphics[width=.5\textwidth]{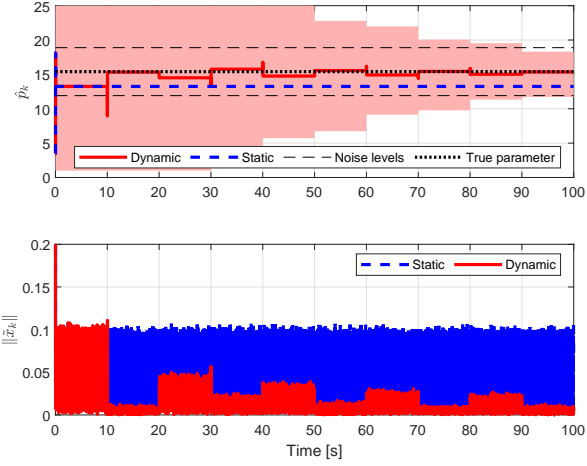}
	\caption{Parameter estimate (top) and norm of the state estimation error (bottom) using the static (dashed) and dynamic sampling policy (solid). The transparent regions indicate the set being sampled by the dynamic sampling policy. The dotted black line indicates the true parameter and the dashed black lines are the noise levels.}
	\label{fig:results}
\end{figure}

Both the static and dynamic sampling policy are implemented using $N=10$ equidistant parameter samples, i.e., $\hat{p}_i = \underline{p}+\frac{\bar{p}-\underline{p}}{2N} + (i-1)\frac{\bar{p}-\underline{p}}{N}$, where $\underline{p}$ and $\bar{p}$ denote the extrema of the set that is currently being sampled (for the dynamic scheme $\underline{p}$ and $\bar{p}$ will move closer together over time). For this sampling scheme, we can guarantee~\eqref{eq:dynamic-sampling} with $\rho(\Delta_m,N) = \frac{\Delta_m}{N}$, as the distance between the true parameter and the nearest sample never exceeds half the distance between neighbouring samples. This also guarantees~\eqref{eq:sampling-scheme} for the static sampling since $\rho(\Delta_0,N)\rightarrow 0$ as $N\rightarrow\infty$. We simulate both the static and the dynamic schemes with design parameters $\lambda=0.995$, $N_{pe} = M_d = 1\cdot 10^{3}$ (which corresponds to $10$ seconds) and $\alpha = 0.8$. The resulting parameter estimate and norm of the state error are shown in Fig.~\ref{fig:results} together with the shrinking parameter set and estimated noise level $\upsilon_{\tilde{p}}(\Delta_v)+\omega_{\tilde{p}}(\Delta_w)$ with $\Delta_v=\Delta_w=0.01$. 

Figure~\ref{fig:results} shows that for both schemes the parameter estimates as well as the state estimate converge within a certain margin of their true value. As can be seen in Fig.~\ref{fig:results}, the first resampling occurs after $10$ seconds, which causes the parameter estimate to jump. The spikes in the parameter estimation error at the switching instances are a result of the monitoring signals being reset, which may cause the supervisor to select a "suboptimal" observer temporarily. Figure~\ref{fig:results} also shows that the estimates do not necessarily become more accurate after individual zoom-ins, which is explained by the fact that if one parameter sample happens to be very accurate, it is not necessarily preserved during the resampling. It should be noted that the number of observers $N=10$ used here is significantly less than the theoretical estimates. To be more specific, our estimates dictate that at least $N^\star=357$ observers are required in the static sampling policy to guarantee that the parameter converges to within $\left\|\tilde{p}_k\right\| \leqslant 3+\upsilon_{\tilde{p}}(\Delta_v)+\omega_{\tilde{p}}(\Delta_w)$. However, the simulations show that this estimate is conservative and that the margin is already achieved for $N=10$. For the dynamic sampling policy, the estimated required number of observers decreases to $N^\star=94$, which is still conservative, however, it confirms that the dynamic scheme requires fewer observers to guarantee similar accuracy.

\section{Conclusions}\label{sec:conclusions}
In this paper, we presented two schemes to jointly estimate parameters and states of discrete-time nonlinear systems in the presence of bounded noise. The first scheme utilizes a static sampling policy and the second scheme uses a dynamic sampling policy. For both schemes, convergence guarantees are provided, which also show that the dynamic scheme typically requires lower computational effort. These results were illustrated by means of a numerical example.

Future work is directed towards obtaining, in an easy and non-conservative manner, the estimates for the required number of observers, time until convergence and minimum time between subsequent zoom-ins on the parameter space needed to get the guarantees as provided in our main theorems. One concrete direction could be the DIRECT algorithm proposed in~\cite{Chong2017}, which eliminates the need to estimate the number of observers a-priori and can be extended to the noisy case to overcome one of these drawbacks. Obtaining non-conservative estimates of the contribution of the noise on the convergence margins is another important research topic. Extending the framework to stochastic noise assumptions may improve performance at the cost of not having worst-case convergence guarantees. Finally, allowing for slowly time-varying parameters, such as in~\cite{Cuevas2019}, in our estimation framework is an interesting future research direction.

\section*{Appendix}
\begin{proof}[Proof of Theorem~\ref{thm:static-sampling}]
	To prove Theorem~\ref{thm:static-sampling}, the following lemma is useful.
\begin{lemma}\label{lem:ISS}
	Consider the system~\eqref{eq:system} and the state estimation error system~\eqref{eq:error-system} under Assumptions~\ref{ass:bounded-sequences}-\ref{ass:iss-lyap}. For any $\Delta_{\tilde{x}},\Delta_x,\Delta_u,\Delta_v,\Delta_w\in\mathbb{R}_{\geqslant 0}$, there exist functions $\beta\in\pazocal{KL}$ and $\gamma_{\tilde{p}},\gamma_{v},\gamma_w\in\pazocal{K}_{\infty}$ such that for any $\tilde{p}_i\in\mathbb{D}$, $i\in\pazocal{N}$, $\tilde{x}_{i,0}\in\mathbb{B}_{\Delta_{\tilde{x}}}^{n_x}$, $p\in\mathbb{P}$, $x_0\in\mathbb{B}_{\Delta_x}^{n_x}$, $\left\{v_k\right\}_{k\in\mathbb{N}}$ with $v_k\in\mathbb{B}_{\Delta_v}^{n_v}$, $\left\{w_k\right\}_{k\in\mathbb{N}}$ with $w_k\in\mathbb{B}_{\Delta_w}^{n_w}$ and $\left\{u_k\right\}_{k\in\mathbb{N}}$ with $u_k\in\mathbb{B}_{\Delta_u}^{n_u}$, $k\in\mathbb{N}$, the corresponding solutions satisfy for $k\in\mathbb{N}$
	\begin{align}
		\left\|\tilde{x}_{i,k}\right\| &\leqslant \beta\left(\left\|\tilde{x}_{i,0}\right\|,k\right) + \gamma_{\tilde{p}}\left(\left\|\tilde{p}_i\right\|\right)\nonumber\\
		&\quad + \gamma_v\left(\left\|\left\{v_j\right\}\right\|\right) + \gamma_{w}\left(\left\|\left\{w_j\right\}\right\|\right).
		\label{eq:ISS}
	\end{align}
\end{lemma}
\begin{proof}[Proof of Lemma~\ref{lem:ISS}]
Given $\Delta_{\tilde{x}},\Delta_x,\Delta_u,\Delta_v,\Delta_w\in\mathbb{R}_{\geqslant 0}$, by Assumption~\ref{ass:bounded-state} there exists $K_x\in\mathbb{R}_{>0}$ such that $x_k\in\mathbb{B}_{K_x}^{n_x}$ for all $k\in\mathbb{N}$. Let $V\,:\,\mathbb{D}\times\mathbb{R}^{n_x}\rightarrow\mathbb{R}_{\geq 0}$, $\sigma\,:\,\mathbb{D}\times\mathbb{R}^{n_v}\times\mathbb{R}^{n_w}\times\mathbb{R}^{n_x}\times\mathbb{R}^{n_u}\rightarrow\mathbb{R}_{\geq 0}$ and $\alpha_1,\alpha_2,\alpha_3\in\pazocal{K}_\infty$ be the functions introduced in Assumption~\ref{ass:iss-lyap}. Since $\sigma$ is continuous and non-negative with $\sigma\left(0,0,0,x,u\right)=0$ for all $x\in\mathbb{R}^{n_x}$ and $u\in\mathbb{R}^{n_u}$, the function $\sigma$ can always be upper bounded by a continuous positive-definite function $\tilde{\sigma}\,:\,\mathbb{D}\times\mathbb{R}^{n_v}\times\mathbb{R}^{n_w}\rightarrow\mathbb{R}_{\geq 0}$ which satisfies $\tilde{\sigma}\left(\tilde{p},v,w\right) \geqslant \max_{x\in\mathbb{B}_{K_x}^{n_x},u\in\mathbb{B}_{\Delta_u}^{n_u}}\sigma\left(\tilde{p},v,w,x,u\right)$ where the inequality accounts for the fact that $\max_{x\in\mathbb{B}_{K_x}^{n_x},u\in\mathbb{B}_{\Delta_u}^{n_u}}\sigma\left(\tilde{p},v,w,x,u\right)$ is not necessarily positive definite. By~\cite[Lemma 3.5]{Khalil1996} and since all norms are equivalent on finite-dimensional vector spaces, there exists $\bar{\sigma}\in\pazocal{K}_{\infty}$ such that
\begin{equation}
	\bar{\sigma}\left(\left\|\left(\tilde{p},v,w\right)\right\|\right)\geqslant \tilde{\sigma}\left(\tilde{p},v,w\right).
	\label{eq:Kinf-supply}
\end{equation}

Substitute~\eqref{eq:Kinf-supply} in~\eqref{eq:diffV} to find $
		V\left(\hat{p},\tilde{x}^+\right) - V\left(\hat{p},\tilde{x}\right)\leqslant -\alpha_3\left(\left\|\tilde{x}\right\|\right) + \bar{\sigma}\left(\left\|\left(\tilde{p},v,w\right)\right\|\right)$. Consider the sequence $\left\{\left(\tilde{p}_i,v_j,w_j\right)\right\}_{j\in\mathbb{N}}$ with $\left(\tilde{p}_i,v_j,w_j\right)\in\mathbb{R}^{n_p+n_v+n_w}$, by~\cite[Lemma 3.5]{Jiang2001}, there exist $\beta\in\pazocal{KL}$ and $\gamma\in\pazocal{K}_{\infty}$ such that
\begin{equation}
	\left\|\tilde{x}_{i,k}\right\| \leqslant \beta\left(\left\|\tilde{x}_{i,0}\right\|,k\right) + \gamma\left(\left\|\left\{\left(\tilde{p}_i,v_j,w_j\right)\right\}\right\|\right).
\end{equation}
Using the triangle inequality for $\pazocal{K}_\infty$-functions~\cite{Khalil1996} and $\left\|\left\{\left(\tilde{p}_i,v_j,w_j\right\}\right)\right\|\leqslant \left\|\tilde{p}_i\right\| + \left\|\left\{v_j\right\}\right\| + \left\|\left\{w_j\right\}\right\|$, we obtain~\eqref{eq:ISS} with $\gamma_{\tilde{p}}\left(r\right)=\gamma_v\left(r\right)=\gamma_w\left(r\right) = \gamma\left(3r\right)$.
\end{proof}

To show that the state estimation error is bounded, i.e., there exists $K_{\tilde{x}}\in\mathbb{R}_{>0}$ such that $\tilde{x}_{i,k}\in\mathbb{B}_{K_{\tilde{x}}}^{n_x}$ for $k\in\mathbb{N}$, recall that $\tilde{p}_i\in\mathbb{D}$, $i\in\pazocal{N}$, with $\mathbb{D}$ being compact and, hence, there exists a constant $\Delta\in\mathbb{R}_{>0}$ such that $\mathbb{D}\subseteq\mathbb{B}_{\Delta}^{n_p}$. Let $\Delta_{\tilde{x}},\Delta_x,\Delta_u,\Delta_v,\Delta_w\in\mathbb{R}_{\geqslant 0}$ be given. By Lemma~\ref{lem:ISS}, for all $\tilde{p}_i\in\mathbb{D}$, $i\in\pazocal{N}$, $\tilde{x}_{i,0}\in\mathbb{B}_{\Delta_{\tilde{x}}}^{n_x}$, $p\in\mathbb{P}$, $x_0\in\mathbb{B}_{\Delta_x}^{n_x}$, $\left\{v_k\right\}_{k\in\mathbb{N}}$ with $v_k\in\mathbb{B}_{\Delta_v}^{n_v}$, $\left\{w_k\right\}_{k\in\mathbb{N}}$ with $w_k\in\mathbb{B}_{\Delta_w}^{n_w}$ and $\left\{u_k\right\}_{k\in\mathbb{N}}$ with $u_k\in\mathbb{B}_{\Delta_u}^{n_u}$, $k\in\mathbb{N}$, the solutions to~\eqref{eq:system} and~\eqref{eq:error-system} satisfy
\begin{equation}
		\left\|\tilde{x}_{i,k}\right\| \leqslant \beta\left(\Delta_{\tilde{x}},0\right) + \gamma_{\tilde{p}}\left(\Delta\right) +\gamma_v\left(\Delta_v\right) + \gamma_w\left(\Delta_w\right),
	\label{eq:bound-x-tilde}%
\end{equation}
with $\beta\in\pazocal{KL}$ and $\gamma_{\tilde{p}},\gamma_v,\gamma_w\in\pazocal{K}_\infty$. Thus, $\tilde{x}_{i,k}\in\mathbb{B}_{K_{\tilde{x}}}^{n_x}$ for all $k\in\mathbb{N}$ with $K_{\tilde{x}}\coloneqq \beta\left(\Delta_{\tilde{x}},0\right) + \gamma_{\tilde{p}}\left(\Delta\right)+\gamma_v\left(\Delta_v\right) + \gamma_w\left(\Delta_w\right)$. 

The following lemma is used in proving Theorem~\ref{thm:static-sampling}.
\begin{lemma}\label{lem:monitoring-bounds}
		Consider the system~\eqref{eq:system}, the state estimation error system~\eqref{eq:error-system}, the monitoring signals~\eqref{eq:monitoring-signals} and Assumptions~\ref{ass:bounded-sequences}-\ref{ass:PE}. For any $\Delta_{\tilde{x}},\Delta_x,\Delta_u,\Delta_v,\Delta_w\in\mathbb{R}_{\geqslant 0}$ and $\epsilon\in\mathbb{R}_{>0}$, there exist functions $\underline{\chi},\bar{\chi}_{\tilde{p}},\bar{\chi}_{v},\bar{\chi}_w\in\pazocal{K}_\infty$ and an integer $M\in\mathbb{N}_{\geqslant 1}$ such that for any $\tilde{p}_i\in\mathbb{D}$, $i\in\pazocal{N}$, $\tilde{x}_{i,0}\in\mathbb{B}_{\Delta_{\tilde{x}}}^{n_x}$, $p\in\mathbb{P}$, $x_0\in\mathbb{B}_{\Delta_x}^{n_x}$, $\left\{v_k\right\}_{k\in\mathbb{N}}$ with $v_k\in\mathbb{B}_{\Delta_v}^{n_v}$, $\left\{w_k\right\}_{k\in\mathbb{N}}$ with $w_k\in\mathbb{B}_{\Delta_w}^{n_w}$ and for some input sequence $\left\{u_k\right\}_{k\in\mathbb{N}}$ with $u_k\in\mathbb{B}_{\Delta_u}^{n_u}$ that satisfies Assumption~\ref{ass:PE}, $k\in\mathbb{N}$, the monitoring signals~\eqref{eq:monitoring-signals} satisfy for all $k\in\mathbb{N}_{\geqslant M}$
	\begin{align}
		\label{eq:monitoring-bounds}
		\underline{\chi}\left(\left\|\tilde{p}_i\right\|\right) \leqslant \mu_{i,k} &\leqslant \epsilon + \bar{\chi}_{\tilde{p}}\left(\left\|\tilde{p}_i\right\|\right) + \bar{\chi}_v\left(\left\|\left\{v_j\right\}\right\|\right)\nonumber\\
		&\quad + \bar{\chi}_w\left(\left\|\left\{w_j\right\}\right\|\right).
	\end{align}
\end{lemma}
\begin{proof}[Proof of Lemma~\ref{lem:monitoring-bounds}]
	Let $N_{\mathrm{pe}}\in\mathbb{N}_{\geqslant 1}$ and $\alpha_{\tilde{y}}\in\pazocal{K}_{\infty}$ be as specified in Assumption~\ref{ass:PE}. It holds that
	\begin{equation}
		\sum_{j=0}^{k-1}\lambda^{k-1-j}\left\|\tilde{y}_{i,j}\right\|^2 \geqslant \lambda^{N_{\mathrm{pe}}-1}\sum_{j=k-N_{\mathrm{pe}}}^{k-1}\left\|\tilde{y}_{i,j}\right\|^2.
		\label{eq:ineq-sum}
	\end{equation}
	By substituting~\eqref{eq:ineq-sum} into~\eqref{eq:monitoring-signals} and using Assumption~\ref{ass:PE}, the lower bound in~\eqref{eq:monitoring-bounds} is obtained with $\underline{\chi}\left(r\right)\coloneqq \lambda^{N_{\mathrm{pe}}-1}\alpha_{\tilde{y}}\left(\left\|\tilde{p}_i\right\|\right)$ for all $r\geqslant 0$.
	
	The upper bound in~\eqref{eq:monitoring-bounds} is obtained as follows. Substitution of~\eqref{eq:error-system-output} in~\eqref{eq:monitoring-signals} yields $\mu_{i,k} = \sum_{j=0}^{k-1}\lambda^{k-1-j}\left\|H\left(\tilde{x}_{i,j},x_j,\tilde{p}_i,p,u_j,w_j\right)\right\|^2$, $i\in\pazocal{N}$, for all $k\in\mathbb{N}$. Recall that $\tilde{x}_{i,k}\in\mathbb{B}_{K_x}^{n_x}$ for all $k\in\mathbb{N}$ as shown in~\eqref{eq:bound-x-tilde} and that $h$ is locally Lipschitz. Hence, there exist $\ell_{\tilde{x}},\ell_{\tilde{p}},\ell_{w}\in\mathbb{R}_{>0}$ such that for all $x\in\mathbb{B}_{K_x}^{n_x}$, $\tilde{x}\in\mathbb{B}_{K_{\tilde{x}}}^{n_x}$, $\tilde{p}\in\mathbb{D}$, $p\in\mathbb{P}$, $u\in\mathbb{B}_{\Delta_u}^{n_u}$ and $w\in\mathbb{B}_{\Delta_w}^{n_w}$, it holds that $\left\|H\left(\tilde{x},x,\tilde{p},p,u,w\right)\right\| = \left\|h\left(\tilde{x}+x,\tilde{p}+p,u,0\right) - h\left(x,p,u,w\right)\right\| \leqslant \ell_{\tilde{x}}\left\|\tilde{x}\right\| + \ell_{\tilde{p}}\left\|\tilde{p}\right\| + \ell_w\left\|w\right\|$. Since $\left(a+b+c\right)^2\leq \left(1+\eta_1^{-1}\right)\left(1+\eta_2^{-1}\right)a^2 + \left(1+\eta_1\right)b^2 + \left(1+\eta_1^{-1}\right)\left(1+\eta_2\right)c^2$, for any $a,b,c,\eta_1,\eta_2\in\mathbb{R}_{>0}$, the monitoring signals satisfy
	\begin{subequations}
		\begin{alignat}{2}
			\mu_{i,k} &\leqslant \sum_{j=0}^{k-1} \lambda^{k-1-j}\left(L_{\tilde{x}}\left\|\tilde{x}_{i,j}\right\|^2+ L_{\tilde{p}}\left\|\tilde{p}_i\right\|^2 + L_w\left\|w_j\right\|^2\right),\nonumber\\
			&\leqslant \frac{L_{\tilde{p}}\left\|\tilde{p}_i\right\|^2+L_w\left\|\left\{w_j\right\}\right\|^2}{1-\lambda}+\sum_{j=0}^{k-1}\lambda^{k-1-j}L_{\tilde{x}}\left\|\tilde{x}_{i,j}\right\|^2,\nonumber
		\end{alignat}
	\end{subequations}
	where $L_{\tilde{x}} = (1+\eta_1^{-1})(1+\eta_2^{-1})l_{\tilde{x}}^2$, $L_{\tilde{p}} = (1+\eta_1)l_{\tilde{p}}^2$ and $L_w = (1+\eta_1^{-1})(1+\eta_2)l_w^2$ with $\eta_1,\eta_2\in\mathbb{R}_{>0}$. Applying Lemma~\ref{lem:ISS}, it holds for all $\tilde{p}_i\in\mathbb{D}$, $i\in\pazocal{N}$, $\tilde{x}_{i,0}\in\mathbb{B}_{\Delta_{\tilde{x}}}^{n_x}$, $p\in\mathbb{P}$, $x_0\in\mathbb{B}_{\Delta_x}^{n_x}$, $\left\{v_k\right\}_{k\in\mathbb{N}}$ with $v_k\in\mathbb{B}_{\Delta_v}^{n_v}$, $\left\{w_k\right\}_{k\in\mathbb{N}}$ with $w_k\in\mathbb{B}_{\Delta_w}^{n_w}$ and $\left\{u_k\right\}_{k\in\mathbb{N}}$ with $u_k\in\mathbb{B}_{\Delta_u}^{n_u}$, $k\in\mathbb{N}$, that
	\begin{align}
		\mu_{i,k} &\leqslant \frac{1}{1-\lambda}\left(L_{\tilde{p}}\left\|\tilde{p}_i\right\|^2 + L_w\left\|\left\{w_j\right\}\right\|^2\right.\nonumber\\
			&\quad + \left. 4L_{\tilde{x}}\left(\gamma_{\tilde{p}}^2\left(\left\|\tilde{p}_i\right\|\right)+ \gamma_v^2\left(\left\|\left\{v_j\right\}\right\|\right) + \gamma_w^2\left(\left\|\left\{w_j\right\}\right\|\right)\right)\right)\nonumber\\
			&\quad + 4L_{\tilde{x}}\sum_{j=0}^{k-1}\lambda^{k-1-j}\beta^2\left(\Delta_{\tilde{x}},j\right).
			\label{eq:mon-up-bnd}
	\end{align}
	
	Given $\epsilon\in\mathbb{R}_{>0}$, choose $\epsilon_{\tilde{x}},\epsilon_{\mu}\in\mathbb{R}_{>0}$ sufficiently small such that $4L_{\tilde{x}}(\epsilon_{\mu} + \frac{\epsilon_{\tilde{x}}}{1-\lambda}) \leqslant \epsilon$ and let $N_{\tilde{x}}\in\mathbb{N}$ be sufficiently large such that $\beta^2\left(\Delta_{\tilde{x}},k\right)\leqslant \epsilon_{\tilde{x}}$ for all $k\in\mathbb{N}_{\geqslant N_{\tilde{x}}}$. Moreover, take $N_{\mu}\in\mathbb{N}_{\geqslant N_{\tilde{x}}}$ sufficiently large such that $\lambda^{k-N_{\tilde{x}}}\frac{1-\lambda^{N_{\tilde{x}}}}{1-\lambda}\beta^2\left(\Delta_{\tilde{x}},0\right)\leqslant \epsilon_{\mu}$ for all $k\in\mathbb{N}_{\geqslant N_{\mu}}$. It follows that for $k\in\mathbb{N}_{\geqslant N_{\tilde{x}}}$
	\begin{align}
		&\sum_{j=0}^{k-1}\lambda^{k-1-j}\beta^2\left(\Delta_{\tilde{x}},j\right) \leqslant \sum_{j=0}^{N_{\tilde{x}}-1} \lambda^{k-1-j}\beta^2\left(\Delta_{\tilde{x}},j\right) + \frac{\epsilon_{\tilde{x}}}{1-\lambda},\nonumber\\
		&\leqslant \frac{\left(\lambda^{k-N_{\tilde{x}}}-\lambda^k\right)\beta^2\left(\Delta_{\tilde{x}},0\right)+\epsilon_{\tilde{x}}}{1-\lambda} \leqslant \frac{\epsilon}{4L_{\tilde{x}}},
	\end{align}
	which is substituted in~\eqref{eq:mon-up-bnd} to obtain the upper bound in~\eqref{eq:monitoring-bounds} with $\bar{\chi}_{\tilde{p}}(r)\coloneqq \frac{1}{1-\lambda}(L_{\tilde{p}}r^2 + 4L_{\tilde{x}}\gamma^2_{\tilde{p}}(r))$, $\bar{\chi}_v(r)\coloneqq \frac{1}{1-\lambda}4L_{\tilde{x}}\gamma^2_v(r)$ and $\bar{\chi}_w(r)\coloneqq \frac{1}{1-\lambda}(L_wr^2 + 4L_{\tilde{x}}\gamma^2_w(r))$ for $r\in\mathbb{R}_{\geqslant 0}$. Finally, it can be concluded that~\eqref{eq:monitoring-bounds} holds for $k\in\mathbb{N}_{\geqslant M}$ with $M\coloneqq\max\left\{N_{\mathrm{pe}},N_{\mu}\right\}$.
\end{proof}

Define $\nu\coloneqq\min\{\frac{1}{3}\nu_{\tilde{p}},\frac{1}{9}\gamma_{\tilde{p}}^{-1}\left(\nu_{\tilde{x}}\right)\}$ where $\gamma_{\tilde{p}}$ is the $\pazocal{K}_{\infty}$-function from Lemma~\ref{lem:ISS} and let $\underline{\chi},\bar{\chi}_{\tilde{p}},\bar{\chi}_v,\bar{\chi}_w\in\pazocal{K}_{\infty}$ be the functions from Lemma~\ref{lem:monitoring-bounds}. Consider (one of) the observer(s) with the smallest parameter estimation error and let $i^\star$ denote its index, i.e., $i^\star \coloneqq \arg\min_{i\in\pazocal{N}} \left\|\tilde{p}_i\right\|$. By definition of the selection criterion~\eqref{eq:selection-criterion}, it holds that $\mu_{\pi_k,k}\leqslant \mu_{i^\star,k}$ for all $k\in\mathbb{N}$ and, hence, by applying Lemma~\ref{lem:monitoring-bounds}, there exists $M\in\mathbb{N}_{\geqslant 1}$ such that, for all $k\in\mathbb{N}_{\geqslant M}$, $\underline{\chi}\left(\left\|\tilde{p}_{\pi_k}\right\|\right) \leqslant \mu_{\pi_k,k} \leqslant \epsilon + \bar{\chi}_{\tilde{p}}\left(\left\|\tilde{p}_i\right\|\right)+ \bar{\chi}_v\left(\left\|\left\{v_j\right\}\right\|\right) + \bar{\chi}_w\left(\left\|\left\{w_j\right\}\right\|\right)$. Here, $\epsilon\in(0,\underline{\chi}(\nu))$ is a design parameter which is chosen by the user. Thus, for all $k\in\mathbb{N}_{\geqslant M}$ it holds that
\begin{equation}
	\left\|\tilde{p}_{\pi_k}\right\| \leqslant \underline{\chi}^{-1}\left(\epsilon+\bar{\chi}_{\tilde{p}}\left(\left\|\tilde{p}_i\right\|\right) + \bar{\chi}_v\left(\left\|\left\{v_j\right\}\right\|\right)+\bar{\chi}_w\left(\left\|\left\{w_j\right\}\right\|\right)\right).
	\label{eq:rand}
\end{equation}
Employ a sampling scheme that satisfies~\eqref{eq:sampling-scheme} to sample the parameter space $\mathbb{P}$. Since $\nu>0$ and $\epsilon<\underline{\chi}\left(\nu\right)$, it is possible to choose the number of samples/observers sufficiently large such that, for any $p\in\mathbb{P}$, it is guaranteed that
\begin{equation}
	\left\|\tilde{p}_{i^\star}\right\| \leqslant \bar{\chi}_{\tilde{p}}^{-1}\left(\underline{\chi}\left(\nu\right) - \epsilon\right).
	\label{eq:sampling-density}
\end{equation}
Let $N^\star\in\mathbb{N}_{\geqslant 1}$ be sufficiently large to ensure~\eqref{eq:sampling-density} for any $N\in\mathbb{N}_{\geqslant N^{\star}}$. By substitution of~\eqref{eq:sampling-density} in~\eqref{eq:rand} and the triangle inequality for $\pazocal{K}_{\infty}$-functions~\cite{Khalil1996}, we find that
\begin{equation}
	\left\|\tilde{p}_{\pi_k}\right\| \leqslant 3\nu + \upsilon_{\tilde{p}}\left(\left\|\left\{v_j\right\}\right\|\right) + \omega_{\tilde{p}}\left(\left\|\left\{w_j\right\}\right\|\right),\label{eq:param-bound}
\end{equation}
for any $N\in\mathbb{N}_{\geqslant N^\star}$. By substituting the definition of $\nu$, we obtain the convergence result for $\tilde{p}_{\pi_k}$ in~\eqref{eq:convergence-static} with $\upsilon_{\tilde{p}}\left(r\right)\coloneqq \underline{\chi}^{-1}\left(3\bar{\chi}_v\left(r\right)\right)$ and $\omega_{\tilde{p}}\left(r\right)\coloneqq \underline{\chi}^{-1}\left(3\bar{\chi}_w\left(r\right)\right)$ for $r\in\mathbb{R}_{\geqslant 0}$.

For $k\in\mathbb{N}_{\geqslant M}$, by~\eqref{eq:param-bound} the supervisor is guaranteed to select from a subset of observers that satisfies, $i\in\bar{\pazocal{N}}\subseteq\pazocal{N}$,
\begin{equation}
		\left\|\tilde{p}_i\right\| \leqslant 3\nu + \upsilon_{\tilde{p}}\left(\left\|\left\{v_j\right\}\right\|\right) + \omega_{\tilde{p}}\left(\left\|\left\{w_j\right\}\right\|\right).
\end{equation} 
By Lemma~\ref{lem:ISS} and the fact that $\nu\leqslant \frac{1}{9}\gamma_{\tilde{p}}^{-1}\left(\nu_{\tilde{x}}\right)$, we find that
	\begin{align}
			\limsup_{k\rightarrow\infty}\left\|\tilde{x}_k\right\| &\leqslant \gamma_{\tilde{p}}\left(3\nu + \upsilon_{\tilde{p}}\left(\left\|\left\{v_j\right\}\right\|\right) + \omega_{\tilde{p}}\left(\left\|\left\{w_j\right\}\right\|\right)\right)\nonumber\\
			&\quad+ \gamma_v\left(\left\|\left\{v_j\right\}\right\|\right) +\gamma_w\left(\left\|\left\{w_j\right\}\right\|\right),
	\end{align}
	from which it can be seen that~\eqref{eq:convergence-static} holds with $\upsilon_{\tilde{x}}(r)\coloneqq \gamma_{\tilde{p}}(3\underline{\chi}^{-1}(3\bar{\chi}_v(r))) + \gamma_v(r)$ and $\omega_{\tilde{x}}(r)\coloneqq \gamma_{\tilde{p}}(3\underline{\chi}^{-1}(3\bar{\chi}_w(r))) + \gamma_w(r)$ for $r\in\mathbb{R}_{\geqslant 0}$. 
\end{proof}

\begin{proof}[Proof of Theorem~\ref{thm:dynamic-sampling}]
	Let $\Delta_{\tilde{x}},\Delta_x,\Delta_u,\Delta_v,\Delta_w\in\mathbb{R}_{\geqslant 0}$, the zooming factor $\alpha\in\left(0,1\right)$ and the desired margins $\nu_{\tilde{x}},\nu_{\tilde{p}}\in\mathbb{R}_{>0}$ be given. Since $\hat{p}_{i,m}\in\mathbb{P}_m$, $i\in\pazocal{N}$ and $m\in\mathbb{N}_{\geqslant 1}$, and, by definition, $\mathbb{P}_m\subseteq\mathbb{P}_0=\mathbb{P}$, it holds that $\hat{p}_{i,m}\in\mathbb{P}$ and $\tilde{p}_{i,m}\in\mathbb{D}$ for any $i\in\pazocal{N}$ and $m\in\mathbb{N}$. Therefore, the results of Lemma~\ref{lem:ISS} and~\ref{lem:monitoring-bounds} apply and, consequently, there exists $K_{\tilde{x}}\in\mathbb{R}_{>0}$ such that $\tilde{x}_{i,k}\in\mathbb{B}_{K_{\tilde{x}}}^{n_x}$ for all $k\in\mathbb{N}$, $i\in\pazocal{N}$, as shown in~\eqref{eq:bound-x-tilde}. 
	
	Define $\bar{\nu} \coloneqq \min\{\nu_{\tilde{p}},\frac{1}{3}\gamma_{\tilde{p}}^{-1}(\nu_{\tilde{x}})\}$, consider the $\pazocal{K}_{\infty}$-functions $\underline{\chi},\bar{\chi}_{\tilde{p}},\bar{\chi}_v,\bar{\chi}_w$ from Lemma~\ref{lem:monitoring-bounds} and let $\bar{\Delta}\in\left(0,\Delta_0\right)$ be sufficiently small such that $3\bar{\chi}_{\tilde{p}}(\rho(\bar{\Delta},0)) < \underline{\chi}(\bar{\nu})\leqslant \underline{\chi}(\nu_{\tilde{p}})$, which is always possible since $\underline{\chi},\bar{\chi}_{\tilde{p}}\in\pazocal{K}_{\infty}$ and since $\rho$ is the $\pazocal{KL}$-function in~\eqref{eq:dynamic-sampling}. Choose $\epsilon\in(0,\frac{1}{3}\underline{\chi}(\alpha\bar{\Delta}))$ sufficiently small such that
	\begin{equation}
		3\bar{\chi}_{\tilde{p}}\left(\rho\left(\bar{\Delta},0\right)\right) + 3\epsilon \leqslant \underline{\chi}\left(\bar{\nu}\right) \leqslant \underline{\chi}\left(\nu_{\tilde{p}}\right),
		\label{eq:case1}
	\end{equation}
	and let $M^\star\in\mathbb{N}_{\geqslant 1}$ be sufficiently large such that the results of Lemma~\ref{lem:monitoring-bounds} hold for all $k\in\mathbb{N}_{\geqslant M^\star}$ using this particular choice of $\epsilon$. Take $N^\star\in\mathbb{N}_{\geqslant 1}$ sufficiently large such that 
	\begin{equation}
		3\bar{\chi}_{\tilde{p}}\left(\rho\left(d,N\right)\right)+3\epsilon \leqslant \underline{\chi}\left(\alpha d\right),\quad \forall d\in\left[\bar{\Delta},\Delta_0\right],
		\label{eq:case2}
	\end{equation}
	for $N\in\mathbb{N}_{\geqslant N^\star}$, which is always possible since $\epsilon<\frac{1}{3}\underline{\chi}\left(\alpha \bar{\Delta}\right)\leqslant \underline{\chi}\left(\alpha d\right)$ for all $d\in\left[\bar{\Delta},\Delta_0\right]$. By Lemma~\ref{lem:monitoring-bounds}, it holds for $M_d \geqslant \max\left\{1,M^\star\right\}$, i.e., by~\eqref{eq:intersampling-time} $k_{m+1} \geqslant k_m + \max\left\{1,M^\star\right\}$, that
	\begin{align}
		\underline{\chi}\left(\left\|\tilde{p}_{i,m}\right\|\right)\leqslant \mu_{i,k_{m+1}} &\leqslant \epsilon + \bar{\chi}_{\tilde{p}}\left(\left\|\tilde{p}_{i,m}\right\|\right)+ \bar{\chi}_v\left(\left\|\left\{v_j\right\}\right\|\right)\nonumber\\
		&\quad + \bar{\chi}_w\left(\left\|\left\{w_j\right\}\right\|\right),
		\label{eq:moni-bnd}
	\end{align}
	for $i\in\pazocal{N}$ and $m\in\mathbb{N}$. Let $i^\star_m\in\arg\min_{i\in\pazocal{N}}\left\|\tilde{p}_{i,m}\right\|$, $m\in\mathbb{N}$, denote the index of (one of) the observer(s) with the smallest parameter estimation error for $k\in\mathbb{N}_{\left[k_m,k_{m+1}-1\right]}$. By definition of $\pi_k$~\eqref{eq:selection-criterion}, it holds that $\mu_{\pi_k,k} \leqslant \mu_{i^\star_m,k}$	for all $k,m\in\mathbb{N}$ and, using~\eqref{eq:moni-bnd} and~\eqref{eq:dynamic-parameter-estimate}, it holds that
	\begin{align}
			\|\tilde{p}_{k_{m+1}}\| &\leqslant \underline{\chi}^{-1}\left(3\epsilon+3\bar{\chi}_{\tilde{p}}\left(\rho\left(\Delta_m,N\right)\right)\right) + \upsilon_{\tilde{p}}\left(\left\|\left\{v_j\right\}\right\|\right)\nonumber\\
			&\quad + \omega_{\tilde{p}}\left(\left\|\left\{w_j\right\}\right\|\right),
			\label{eq:thm2-partial}%
		\end{align}
	where $\upsilon_{\tilde{p}}\left(r\right)\coloneqq \underline{\chi}^{-1}\left(3\bar{\chi}_v\left(r\right)\right)$ and $\omega_{\tilde{p}}\left(r\right)\coloneqq \underline{\chi}^{-1}\left(3\bar{\chi}_w\left(r\right)\right)$ for $r\in\mathbb{R}_{\geqslant 0}$. Based on~\eqref{eq:thm2-partial}, two cases are distinguished. First, if $\Delta_m\leqslant \bar{\Delta}$ it follows from~\eqref{eq:case1} that $\|\tilde{p}_{k_{m+1}}\| \leqslant \bar{\nu} + \upsilon_{\tilde{p}}\left(\left\|\left\{v_j\right\}\right\|\right) + \omega_{\tilde{p}}\left(\left\|\left\{w_j\right\}\right\|\right)$. Second, if $\Delta_m > \bar{\Delta}$ it holds that $\Delta_m\in\left[\bar{\Delta},\Delta_0\right]$ and by~\eqref{eq:case2}, we have $\|\tilde{p}_{k_{m+1}}\|\leqslant \Delta_{m+1} + \upsilon_{\tilde{p}}\left(\left\|\left\{v_j\right\}\right\|\right)+ \omega_{\tilde{p}}\left(\left\|\left\{w_j\right\}\right\|\right)$, where $\Delta_{m+1}=\alpha\Delta_m$. Thus, $p\in\mathbb{P}_{m+1}$ with $\mathbb{P}_{m+1}$ as defined in~\eqref{eq:shrinking-parameter-sets}. Since $\Delta_m = \alpha^m\Delta_0\rightarrow 0$ as $m\rightarrow\infty$, there exists $\bar{M}\in\mathbb{N}_{\geqslant 1}$ such that
	\begin{equation}	
		\left\|\tilde{p}_k\right\| \leqslant \bar{\nu} +\upsilon_{\tilde{p}}\left(\left\|\left\{v_j\right\}\right\|\right)+ \omega_{\tilde{p}}\left(\left\|\left\{w_j\right\}\right\|\right),\label{eq:thm2-param-bound}
	\end{equation}
	for all $k\in\mathbb{N}_{\geqslant \bar{M}}$. Substitute $\bar{\nu}\leqslant \nu_{\tilde{p}}$ to obtain the convergence result for $\tilde{p}_k$ in~\eqref{eq:convergence-dynamic}. From Lemma~\ref{lem:ISS} it follows that $\limsup_{k\rightarrow\infty}\left\|\tilde{x}_k\right\|\leqslant \gamma_{\tilde{p}}\left(\left\|\tilde{p}_k\right\|\right) + \gamma_v\left(\left\|\left\{v_j\right\}\right\|\right) + \gamma_w\left(\left\|\left\{w_j\right\}\right\|\right)$ where $\left\|\tilde{p}_k\right\|$ satisfies~\eqref{eq:thm2-param-bound} for all $k\in\mathbb{N}_{\geqslant \bar{M}}$. Recall that $\bar{\nu} \leqslant \frac{1}{3}\gamma_{\tilde{p}}^{-1}\left(\nu_{\tilde{x}}\right)$ and $\alpha\left(a+b+c\right)\leqslant \alpha\left(3a\right) + \alpha\left(3b\right)+\alpha\left(3c\right)$ for any $\alpha\in\pazocal{K}_{\infty}$ and $a,b,c\in\mathbb{R}_{\geqslant 0}$, which leads to the convergence result for $\tilde{x}_{k}$ in~\eqref{eq:convergence-dynamic} with $\upsilon_{\tilde{x}}(r)\coloneqq \gamma_v\left(r\right) + \gamma_{\tilde{p}}\left(3\upsilon_{\tilde{p}}\left(r\right)\right)$ and $\omega_{\tilde{x}}\left(r\right)\coloneqq \gamma_w\left(r\right) + \gamma_{\tilde{p}}\left(3\omega_{\tilde{p}}\left(r\right)\right)$, $r\in\mathbb{R}_{\geqslant 0}$.
\end{proof}

\begin{proof}[Proof of Proposition~\ref{prop:circle-criterion-observer}]
	By~\eqref{eq:lipschitz} and the mean value theorem, there exist $\delta_i\in\left[0,l_i\right]$, $i\in\mathbb{N}_{\left[1,n_\phi\right]}$ such that $\phi_i\left(z+q\right)-\phi_i\left(q\right)=\delta_{i}z$ for all $z,q\in\mathbb{R}$. Hence, for $z=\bm{H}(\hat{p})\tilde{x}-K(\hat{p})w$ and $q=Hx$ we have, for any $\tilde{x},x\in\mathbb{R}^{n_x}$ and $w\in\mathbb{R}^{n_w}$, that $\phi(z+q)-\phi(q)=\delta z$ with $\delta = \mathrm{diag}(\delta_1,\hdots,\delta_{n_\phi})$. Let $\xi = (\tilde{x},\delta z,v,w)$ and
	\begin{equation}
		\tilde{x}^+ = \pazocal{Q}(\hat{p},p)\xi + \psi(\tilde{p},p,x,u),
		\label{eq:err-diff-eq}
	\end{equation}
	 with $\pazocal{Q}(\hat{p},p)=\begin{bmatrix}\begin{smallmatrix}\bm{A}(\hat{p})&G(\hat{p}) &-B(p) & -L(\hat{p})\end{smallmatrix}\end{bmatrix}$ and $\psi(\tilde{p},p,x,u)\coloneqq (G(\tilde{p}+p)-G(p))\phi(Hx) + (A(\tilde{p}+p)-A(p))x + (B(\tilde{p}+p)-B(p))u$. Let $V(\tilde{x}) = \tilde{x}^\top P\tilde{x}$, with $P=P^\top\succ 0$, $a_1 = \lambda_{\mathrm{min}}\left(P\right)$ and $a_2=\lambda_{\mathrm{max}}\left(P\right)$, then condition~\eqref{eq:sandwich-bounds} holds with $\alpha_i(r) = a_ir^2$, $i=1,2$. Substitute~\eqref{eq:err-diff-eq} in $V(\tilde{x})$ to find
	\small\begin{align}
			V\left(\tilde{x}^+\right) &= \xi^\top\pazocal{Q}^{\top}(\hat{p},p)P\pazocal{Q}(\hat{p},p)\xi\nonumber\\
			&\quad +2\psi^\top P\left(\tilde{x}^+-\psi\right) + \psi^\top P\psi,\\
			&\leqslant 2\xi^\top\pazocal{Q}^\top(\hat{p},p)P\pazocal{Q}(\hat{p},p)\xi + 2\psi^\top P\psi\label{eq:lyap-desc}
	\end{align}\normalsize
	where the arguments of $\psi$ have been omitted. Since $P\succ 0$, the LMI~\eqref{eq:design-inequality} implies, using Schur complement,
\begin{equation}
	\pazocal{Q}^\top(\hat{p},p)P\pazocal{Q}(\hat{p},p)\preccurlyeq \begin{bmatrix}\begin{smallmatrix}
		\frac{1}{2}P-\frac{\kappa_{\tilde{x}}}{2}I & \star & \star & \star\\
		-\frac{1}{2}M\bm{H}(\hat{p}) & M\Lambda^{-1} & \star & \star\\
		0 & 0 & \frac{\kappa_v}{2}I & \star\\
		0 & \frac{1}{2}K^\top(\hat{p})M & 0 & \frac{\kappa_w}{2}I
	\end{smallmatrix}\end{bmatrix},\nonumber
\end{equation}
which is substituted in~\eqref{eq:lyap-desc} to obtain
\begin{align}
	&V(\tilde{x}^+) \leqslant V(\tilde{x}) - \kappa_{x}\left\|\tilde{x}\right\|^2_2 + \kappa_v\left\|v\right\|^2_2 + \kappa_w\left\|w\right\|_2^2\nonumber\\
	&\quad + 2z^\top\delta M\left(\Lambda^{-1}\delta z - \bm{H}(\hat{p})\tilde{x}\right) + 2\psi^\top P\psi,
\end{align} 
It can be shown that $z^\top\delta M\left(\Lambda^{-1}\delta z-\bm{H}(\hat{p})\tilde{x}\right)= z^\top\delta M\left(\Lambda^{-1}\delta - I\right)z\leqslant 0$ by examining the entries of $\Lambda^{-1}\delta$. Thus, condition~\eqref{eq:diffV} holds with $\alpha_3(r)\coloneqq \kappa_{\tilde{x}}r^2$ and $\sigma(\tilde{p},v,w,x,u)\coloneqq \kappa_v\left\|v\right\|^2_2 + \kappa_w\left\|w\right\|_2^2 + 2\sup_{p\in\mathbb{P}}\psi^\top(\tilde{p},p,x,u)P(\tilde{p}+p)\psi(\tilde{p},p,x,u),$ which is continuous, since $A$, $G$, $B$ and $\phi$ are continuous and due to Assumptions~\ref{ass:bounded-sequences}-\ref{ass:bounded-state}. Moreover, $\sigma$ is non-negative and $\sigma(0,0,0,x,u)=0$ for all $x\in\mathbb{R}^{n_x}$, $u\in\mathbb{R}^{n_u}$. Since the state error for system~\eqref{eq:illustrative-example} with observer~\eqref{eq:circle-criterion-observer} satisfies~\eqref{eq:err-diff-eq} with $\tilde{x}^+=\tilde{x}_{k+1}$, $\xi = \left(\tilde{x}_k,\bm{\delta}_kz_k,v_k,w_k\right)$ where $\bm{\delta}_k = \mathrm{diag}(\delta_{1,k},\hdots,\delta_{n_\phi,k})$ with $\delta_{i,k}$ such that $\phi_i(z_k+Hx_k)-\phi_i(Hx_k) = \delta_{i,k}z_k$, $i\in\mathbb{N}_{\left[1,n_{\phi}\right]}$, for $k\in\mathbb{N}$, Assumption~\ref{ass:iss-lyap} is satisfied.
\end{proof}

\bibliographystyle{IEEEtran}
\bibliography{../../../../bibtex/phd-bibtex}

\addtolength{\textheight}{-3cm}   
                                  
\end{document}